\documentclass{amsart}

\usepackage[latin1]{inputenc}
\usepackage{amsmath,amssymb}

\usepackage{color}
\usepackage{tikz}
\usepackage[unicode=true]{hyperref}

\newtheorem{thm}{Theorem}[section]
\newtheorem{cor}[thm]{Corollary}
\newtheorem{prop}[thm]{Proposition}
\newtheorem{lem}[thm]{Lemma}
\newtheorem{conj}[thm]{Conjecture}

\theoremstyle{definition}
\newtheorem{defn}[thm]{Definition}

\newtheorem*{rem}{Remark}
\newtheorem*{ex}{Example}

\makeatletter
\newcommand*{\centernot}{\mathpalette\@centernot}
\def\@centernot#1#2{%
  \mathrel{%
    \rlap{%
      \settowidth\dimen@{$\m@th#1{#2}$}%
      \kern.5\dimen@
      \settowidth\dimen@{$\m@th#1=$}%
      \kern-.5\dimen@
      $\m@th#1\not$%
    }%
    {#2}%
  }%
}
\makeatother

\newcommand{\bbA}{\mathbb{A}}
\newcommand{\bbF}{\mathbb{F}}

\newcommand{\bbN}{\mathbb{N}}

\newcommand{\bbQ}{\mathbb{Q}}

\newcommand{\bbZ}{\mathbb{Z}}

\newcommand{\tup}[1]{\underline{#1}}  
\newcommand{\stup}[1]{\hat{\tup{#1}}}  
\newcommand{\atup}{\tup{a}}
\newcommand{\atiltup}{\tup{\tilde{a}}}
\newcommand{\ctup}{\tup{c}}
\newcommand{\vtup}{\tup{v}}
\newcommand{\xtup}{\tup{x}}
\newcommand{\ytup}{\tup{y}}\newcommand{\ystup}{\stup{y}}
\newcommand{\ztup}{\tup{z}}

\newcommand{\kappatup}{\tup{\kappa}}
\newcommand{\mutup}{\tup{\mu}}

\newcommand{\Atil}{\tilde{A}}

\newcommand{\ohne}{\setminus}
\newcommand{\del}{\partial}
\newcommand{\surject}{\twoheadrightarrow}
\newcommand{\inject}{\hookrightarrow}
\newcommand{\iso}{\mathrel{\overset{\sim}{\smash{\longrightarrow}\vrule height.3ex
width0ex\relax}}}

\newcommand{\mult}{^{\times}}
\newcommand{\auf}[1]{\mathord{\upharpoonright_{#1}}}

\newcommand{\ngg}{\centernot\gg}
\newcommand{\nll}{\centernot\ll}

\newcommand{\defiff}{\mathrel:\!\joinrel\iff}

\newcommand{\bifurc}{\mathcal{Y}}

\newcommand{\ccGamma}{\Xi}  

\newcommand{\skel}{\mathcal{S}} 
\newcommand{\fintree}{\mathcal{F}}  
\newcommand{\tree}{\mathcal{T}}
\newcommand{\branch}{\mathcal{B}}  

\newcommand{\Trees}{\{\mathrm{Trees}\}}

\newcommand{\joint}{\tilde{v}}
\newcommand{\bone}{\tilde{e}}

\newcommand{\ifu}{\ell}  

\newcommand{\sv}{Z}  

\newcommand{\Pt}{\{\mathrm{Pt}\}} 

\newcommand{\etwa}[1][\delta]{\approx_{#1}}

\newcommand{\ball}[1]{B(#1)}
\newcommand{\BX}{B_X}
\newcommand{\BY}{B_Y}

\newcommand{\conda}{\vartriangleleft}
\newcommand{\condb}{\mathrel{\rlap{\raisebox{-0.18ex}{$\vartriangleleft$}}\raisebox{0.18ex}{$\vartriangleleft$}}}

\newcommand{\phitree}{\phi_{\mathrm{tree}}}

\newcommand{\pow}{u}
\newcommand{\dmax}{d_{\mathrm{max}}}
\newcommand{\Nshift}{\tilde{N}}
\newcommand{\treeshift}{\tilde{\tree}}

\newcommand{\dQp}{\operatorname{dist}_{\Qp}} 

\newcommand{\dats}{\mathcal{D}}

\newcommand{\GL}{\operatorname{GL}}
\newcommand{\im}{\operatorname{im}}
\newcommand{\Tr}{\operatorname{T}}
\newcommand{\Trtil}{\operatorname{\tilde{T}}}

\newcommand{\depth}{\operatorname{depth}}
\newcommand{\spec}{\operatorname{spec}}

\newcommand{\Zp}{\bbZ_p}
\newcommand{\Qp}{\bbQ_p}

\newcommand{\acl}[1]{\tilde{#1}}
\newcommand{\aclQp}{\acl{\bbQ}_p}
\newcommand{\aclZp}{\acl{\bbZ}_p}
\newcommand{\aclGamma}{\acl{\Gamma}}

\newcommand{\cl}[1]{\bar{#1}}
\newcommand{\clX}{\cl{X}}

\newcommand{\req}[1]{(\ref{eq:#1})}
\newcommand{\rit}[1]{(\ref{it:#1})}

\begin{document}

\title{Trees of definable sets over the $p$-adics}
\author{Immanuel Halupczok}
\thanks{The author was supported by the Fondation Sciences mathématiques
de Paris.}
\begin{abstract}
To a definable subset of $\Zp^n$ (or to a scheme of finite type over $\Zp$)
one can associate a tree in a natural way.
It is known that the corresponding Poincar\'e series
$\sum N_\lambda Z^\lambda \in \bbZ[[Z]]$ is rational,
where $N_\lambda$ is the number of nodes of the tree at depth $\lambda$.
This suggests that the trees themselves are far from arbitrary.
We state a conjectural, purely combinatorial description of the class of possible trees
and provide some evidence for it.
We verify that any tree in our class indeed arises from a definable set,
and we prove that the tree of a definable set (or of a scheme) lies in our class
in three special cases: under weak smoothness assumptions,
for definable subsets of $\Zp^2$, and for one-dimensional sets.
\end{abstract}

\maketitle

\section{Introduction and results}

Suppose that $X \subset \Qp^n$ is a definable set in the language of fields.
For $\lambda \ge 0$, let $X_\lambda$ be the image of $X \cap \Zp^n$
under the projection $\Zp^n \surject (\bbZ/p^\lambda\bbZ)^n$.
In \cite{Den:rat}, Denef proved that the associated Poincar\'e series
\[
P_X(\sv) := \sum_{\lambda = 0}^{\infty} \#X_\lambda \cdot \sv^\lambda
\in \bbZ[[\sv]]
\]
is a rational function in $\sv$. Now the disjoint union
$\Tr(X) := \bigcup_{\lambda \ge 0}X_\lambda$
carries a tree structure defined by the projections
$(\bbZ/p^{\lambda+1}\bbZ)^n \surject (\bbZ/p^\lambda\bbZ)^n$,
thus a natural question (which Loeser posed to me) is:
can the result of Denef be refined to a result about the structure of the trees?
In other words:
does there exist a purely combinatorial description of the structure of
trees which can arise from definable sets, which implies the above
rationality?

The goal of this article is to conjecturally give such a description
and to provide some evidence for it.
More precisely, for any $d \in \bbN$ we define a class of
``trees of level $d$''. Our conjecture is then:

\begin{conj}\label{conj:main}
Suppose that $X \subset \Qp^n$ is a definable set.
Then $\Tr(X)$ is a tree of level $\dim X$.
\end{conj}

Here, the dimension of a definable set $X$ is the algebraic
dimension of the Zariski closure of $X$ in the algebraic closure $\aclQp^n$;
see \cite{SD:dim}.

Whether the conjecture is interesting depends on how tight our definition
of trees of level $d$ is. In fact, we will show that it is as tight
as possible:

\begin{thm}\label{thm:inv}
For any tree $\tree$ of level $d$ without leaves, there exists a definable
set $X$ of dimension $\le d$ such that $\Tr(X) \cong \tree$.
\end{thm}

The tree $\Tr(X)$ of a set never has leaves, so we might as well forbid
leaves in our definition of trees of level $d$; however,
for technical reasons it is better to allow them.

By Theorem~\ref{thm:inv}, our definition of level $d$ trees is clearly
precise enough to imply rationality of the Poincar\'e series.
However, we will also give an easy direct proof
in Proposition~\ref{prop:rat}.

\medskip

The main results of this article are proofs of the conjecture
in several special cases. Before stating these results,
let us consider an algebraic variant of the trees. For any scheme $V$ of
finite type over $\Zp$, we define a tree $\Tr(V)$ as follows:
the set of nodes at depth $\lambda$ is the image of the map
$V(\Zp) \to V(\bbZ/p^\lambda\bbZ)$,
and the tree structure is given by the maps
$V(\bbZ/p^{\lambda+1}\bbZ) \to V(\bbZ/p^\lambda\bbZ)$.
Using this, we can state an algebraic variant of the conjecture:

\begin{conj}\label{conj:alg}
Suppose that $V$ is a scheme of finite type over $\Zp$.
Then $\Tr(V)$ is a tree of level $\dim V$.
\end{conj}

If $V$ is an affine embedded scheme (in $\bbA^n$, say), then
we have $V(\Qp) \subset \Qp^n$, and
the two definitions yield the same tree: $\Tr(V) \cong
\Tr(V(\Qp))$. Once the definition of a level $d$
tree is given, it will be easy to verify that if the conjecture
holds for each set of a finite cover of $V$, then it also holds
for $V$ itself (Proposition~\ref{prop:affin}); thus
Conjecture~\ref{conj:main} implies Conjecture~\ref{conj:alg}.
Therefore in most of the article we will stick to the affine
case and the first definition of trees.

From an algebraic point of view, it seems more natural
to consider a tree $\Trtil(V)$ whose set of nodes at depth $\lambda$ is the
whole set $V(\bbZ/p^\lambda\bbZ)$, and not only
the image of $V(\Zp)$. Indeed, the Poincar\'e series
\begin{equation}\label{eq:algTree}
\sum_{\lambda = 0}^{\infty} \#V(\bbZ/p^\lambda\bbZ) \cdot \sv^\lambda
\end{equation}
is rational,
too, and at the end of this article, we will describe
a variant of the conjecture which includes both kinds of trees
(and much more). However, for now let us stick to the trees $\Tr(V)$.

\medskip

We now present the cases in which we will prove the conjecture.
The first one is not very difficult to prove.
Under rather weak smoothness assumptions, the tree of a scheme
is particularly simple.

\begin{thm}\label{thm:mainSmooth}
Suppose that $V$ is a scheme of finite type over $\Zp$, and
suppose that for every $\Zp$-valued point $x\colon \spec\Zp \to V$,
$V$ is smooth at $x(\eta)$, where $\eta$ is the generic point of $\spec\Zp$.
Then $\Tr(V)$ consists of a finite tree, with copies of $\Tr(\Zp^d)$,
$d \le \dim V$ attached to its leaves ($d$ may depend on the leaf).
In particular, $\Tr(V)$ is a tree of level $\dim V$.
\end{thm}

More generally, if $V$ is a non-smooth scheme, then the tree still looks
like $\Tr(\Zp^d)$ close to any smooth point.
On the other hand, we will see on an example (Subsection~\ref{subsect:cusp})
that close to singular points, the trees do get complicated.
(In fact trees of definable sets are not essentially more complicated than
trees of varieties.)
Thus the information contained in a tree of a scheme describes its
singularities; this should be closely related to the structure of arc
spaces above singularities, as studied in \cite{Nas:arc}.

The more interesting cases of the main conjecture
which we will verify are the following.

\begin{thm}\label{thm:mainZp2}
Conjecture~\ref{conj:main} holds if $X \subset \Qp^2$.
\end{thm}

\begin{thm}\label{thm:main1dim}
Conjecture~\ref{conj:main} holds if $\dim X \le 1$.
\end{thm}

The present proofs of these results crucially rely on the theorem of
Puiseux, which is valid only for curves.
Thus to generalize them to higher dimension, one will need some new ideas.

Let me mention one more reason for which the trees seem interesting to me.
Suppose $X_1$ and $X_2$ are two definable subsets of $\Zp^n$ which are
closed in $p$-adic topology. Then isometric bijections between
$X_1$ and $X_2$ correspond exactly to isomorphisms of the corresponding trees
(see Lemma~\ref{lem:treeIso}). Thus one can interpret trees as a step
towards classification of definable sets up to isometry.
Indeed, if the main conjecture is true, then up to $p$-adic closure
any definable set is isometric to a set of the form constructed
in the proof of Theorem~\ref{thm:inv}.

\medskip

The remainder of this article is organized as follows.

In the next section, we fix our notation.

In Section~\ref{sect:bsp}, we compute the first trees:
we prove Theorem~\ref{thm:mainSmooth} and we give an example
of a tree of a singular curve.
To be able to do that, we first prove a key lemma
(Corollary~\ref{cor:contTree}) which relates the tree of a set to
the trees of its fibers.

The trees of Section~\ref{sect:bsp} give an idea of
how level $d$ trees should look like; in Section~\ref{sect:defTree},
we will actually define them.
We will give two versions of the definition:
a restrictive one and a more relaxed one; then we will show
that both are equivalent. At the end of this section,
we will verify some first properties of level $d$ trees.

In Section~\ref{sect:results}, we will prove statements about
given trees of level $d$, namely Theorem~\ref{thm:inv} and
the rationality of the Poincar\'e series of such a tree.

Section~\ref{sect:proofs} is devoted to the proof of the main conjecture
for subsets of $\Qp^2$ and for one-dimensional sets.
The section starts with a sketch of the principal ideas; then we
introduce the main tools we need, namely cell decomposition and
a way to understand definable functions on small balls.
In Subsection~\ref{subsect:1dimParam}, we 
prove a parametrized version of the conjecture for subsets of $\Qp$,
and finally we finish the actual proofs.

To conclude, we will present some
possible generalizations of the conjecture
in Section~\ref{sect:open}.

\section{Notation}

\subsection{Notation concerning model theory and \texorpdfstring{$\Qp$}{\041\032p}}

We fix a prime $p$ once and for all and work in $\Qp$.
We will use a two-sorted language, with one sort for $\Qp$ and one
for the valuation group $\Gamma$. As usual, we take the ring language
on $\Qp$, the ordered group language on $\Gamma$ and a valuation map $v\colon
\Qp \to \Gamma \cup \{\infty\}$. Note that
$\Gamma$ and $v$ are interpretable in the pure field language
(see e.g.\ \cite{Den:cell}, Lemma~2.1), so using the two-sorted language
is not really different from using the pure field language.

By ``definable'' we will always mean definable with parameters.

We will sometimes identify $\Gamma$ with $\bbZ$. In particular, we will
write $1$ for the valuation of $p$, and we will
often use the cross section $\Gamma \to \Qp\mult, \lambda \mapsto p^\lambda$.

For $X \subset \Qp^n$, we denote by $\clX$ the closure of $X$ in the $p$-adic
topology.

For $\xtup = (x_1, \dots, x_n) \in \Qp^n$ and $\lambda \in \Gamma$, $\ball{\xtup,\lambda}
:= \xtup + p^\lambda \Zp^n$
denotes the ball around $\xtup$ of ``radius'' $\lambda$.
Moreover,
$v(\xtup) := \min\{v(x_i)\mid 1\le i \le \ell\}$ is the
minimum of the valuations of the coordinates. (In other words:
$v(\xtup) \ge \lambda \iff \xtup \in \ball{0,\lambda}$.)
Note that for us a ball always has the same radius in each coordinate.

The following non-standard notation will be very handy:
\begin{defn}
For $\delta \in \Gamma_{>0}$ and $x, x' \in \Qp\mult$,
we write $x\etwa x'$ if $x$ and $x'$ have the same image under the canonical
homomorphism $\Qp\mult \surject \Qp\mult/\ball{1,\delta}$. Equivalently,
\[
x\etwa x' \mathrel{:\joinrel\iff} v(x-x') \ge v(x) + \delta
.
\]
\end{defn}

Occasionally, we will work in the algebraic closure $\aclQp$ of $\Qp$.
Write $\aclZp$ for the valuation ring and $\aclGamma$
for the value group of $\aclQp$. The definitions of $v(\xtup)$ and $x\etwa x'$
also make sense in this context. $1 \in \aclGamma$ will still denote
the valuation of $p$.

Let $e \in \bbN_{\ge1}$.
The \emph{$e$-th power residue} of $x \in \Qp\mult$ is the
set $\{y^e\cdot x \mid y \in \Qp\mult\}$.
The following statements are well known (and not difficult to prove):

\begin{lem}\label{lem:root}
Suppose $e \in \bbN_{\ge1}$.
\begin{enumerate}
\item
If $\delta \ge v(e) + 1$,
then the map $z \mapsto z^e$ induces a bijection $1+p^{\delta}\Zp \to 1+p^{\delta+v(e)}\Zp$.
\item
If $x_1, x_2\in \Qp$ satisfy $x_1 \etwa[2v(e)+1] x_2$, then $x_1$ and $x_2$ have the same
$e$-th power residue.
\item
There are only finitely many different $e$-th power residues.
\end{enumerate}
\end{lem}

%

\subsection{Model theory of \texorpdfstring{$\Gamma$}{\003\223}}

Let $M$ be a subset of $\Gamma^m$. A function $\ifu\colon M \to\Gamma$
is called \emph{linear} if there exist $a_1, \dots, a_m, b \in \bbQ$ such that
$\ifu(\kappa_1, \dots, \kappa_m) = a_1\kappa_1 + \dots a_m\kappa_m + b$
for all $(\kappa_1, \dots, \kappa_m) \in M$. A function $M \to\Gamma \cup \{\infty\}$
is called \emph{linear} if it is either a linear function to $\Gamma$ or constant $\infty$.
We will use the partial order on the functions $M \to\Gamma \cup \{\infty\}$
defined by $\ifu \le \ifu' \defiff \ifu(\kappatup) \le \ifu'(\kappatup)$
for all $\kappatup \in M$.

It is well known that any subset $M \subset \Gamma^{m}$ which is
definable in our
two-sorted structure is already definable in $(\Gamma, 0, +, <)$.
We will use the cell decomposition theorem for that structure (see
e.g.\ \cite{Clu:cell}, Theorem~1) to get hold of definable subsets of
$\Gamma^m$.
To avoid the rather lengthy definition of cells, we only state an
immediate consequence of the cell decomposition theorem.

\begin{lem}\label{lem:cellGamma}
\begin{enumerate}
\item\label{it:linGamma}
For any definable $M \subset \Gamma^{m}$ and any definable
function $\ifu\colon M \to \Gamma$,
there exists a finite partition of $M$ into definable subsets $M'$
such that $\ifu$ is linear on each part $M'$.
\item\label{it:cellGamma}
Any definable subset $N \subset \Gamma^{m} \times \Gamma$ can be written
as a boolean combination of sets of the following forms:
\[
\begin{array}{c@{\quad}l}
M \times \Gamma
&\text{for $M \subset \Gamma^{m}$ definable}
\\
\{(\kappatup, \lambda)\in \Gamma^{m} \times \Gamma \mid \lambda \lesseqqgtr \ell(\kappa)\}
&\text{for $\ell\colon \Gamma^{m} \to \Gamma$ linear}
\\
\Gamma^{m} \times \ccGamma
&\text{for $\ccGamma \in \Gamma/\rho\Gamma$, $\rho \in \Gamma$}
.
\end{array}
\]
\end{enumerate}
\end{lem}

\subsection{Trees and Swiss cheese}

There are different ways to define trees. Let
me fix the variant I will use.

\begin{defn}
A \emph{tree} $\tree$ is a set (of \emph{nodes}), together with a binary
is-child-of relation, which satisfies the usual axioms. However,
we do allow trees to be empty.
Define the \emph{root} (if the tree is non-empty),
the \emph{leaves} and the \emph{depth} $\depth(v) = \depth_{\tree}(v)$ of a node $v\in \tree$ as usual.

We say that $(v, v')$ is an \emph{edge} of $\tree$ if $v'$ is a child of $v$.
A \emph{path} (of length $n$) is a sequence $v_0, \dots, v_n$ of nodes where
$(v_i,v_{i+1})$ are edges.

The class of all trees will be denoted by $\Trees$.

Define isomorphisms of trees as usual. The product $\tree_1 \times \tree_2$
of two trees is defined layerwise.

If $\tree$ and $\tree'$ are two non-empty trees and $v$ is a node of $\tree$, then
we will sometimes construct a new tree by \emph{attaching $\tree'$ to $v$}.
This means: take the disjoint union of the nodes and then
identify the root of $\tree'$ with $v$.
\end{defn}

We already gave a definition of the tree of a set in the introduction.
Here is a slight generalization.

\begin{defn}
Suppose $X \subset \Qp^n$ is a set and
${B_0} = \ball{\xtup_0, \lambda_0} \subset \Qp^n$
a ball. Then the \emph{tree of $X$ on $B_0$} is
\[\Tr_{B_0}(X) := \Tr_{\xtup_0, \lambda_0}(X) :=
\{\ball{\xtup, \lambda} \subset B_0 \mid \ball{\xtup, \lambda} \cap X \ne \emptyset\}
,
\]
with the tree structure induced by inclusion.
Set $\Tr(X) := \Tr_{\Zp}(X)$.
\end{defn}
\begin{rem}
$\Tr_{B_0}(X)$ only depends on $B_0 \cap X$. In particular, $\Tr_{B_0}(X)$
is empty if and only if $B_0 \cap X = \emptyset$.
\end{rem}

\begin{ex}
The tree $\Tr(\Pt)$ of a one-point set is just one infinite path.
$\Tr(\Zp^n)$ is the infinite tree where each node has exactly $p^n$
children.
\end{ex}

One technique to determine the tree $\Tr(X)$ of a definable set will
be to cut out some balls $B_i$ on which $X$ is particularly complicated,
compute the trees $\Tr_{B_i}(X)$ separately,
compute the tree on the remainder,
and then put everything together. We define notation suitable for this.

\begin{defn}
A \emph{slice of Swiss cheese} (or a \emph{cheese}, for short) is a set of the form
$S = B \ohne \bigcup_{i\in I}B_i$, where $I$ is a finite index set and
$B$ and $B_i$ are balls in $\Zp^n$, satisfying $B_i \subset B$
and $B_i \cap B_j = \emptyset$ for $i \ne j$.
The set of balls $B_i$ (the ``holes'') is part of the cheese datum.
\end{defn}

\begin{defn}
Let $S = B_0 \ohne \bigcup_{i\in I}B_i \subset \Zp^n$ be a cheese and
$X \subset \Zp^n$ a set.
Then the \emph{tree $\Tr_S(X)$ of $X$ on $S$} is the subtree of
$\Tr_{B_0}(X)$ consisting of those nodes $B$ which are not
a proper subset of any $B_i$, $i \in I$.
\end{defn}
We will only be interested in the tree $\Tr_S(X)$ when none of the
intersections $X \cap B_i$ is empty. In that case,
the balls $B_i$ are nodes of $\Tr_S(X)$, and
the total tree $\Tr_{B_0}(X)$ can be obtained from $\Tr_S(X)$ by attaching
$\Tr_{B_i}(X)$ to the node $B_i$ for each $i \in I$.

\section{Computing the first trees}
\label{sect:bsp}

The definition of a tree of level $d$ is rather involved, so
let us first compute a few examples to motivate it. To this end,
we first prove some basic lemmas on trees. In particular,
we will check that in certain cases the tree of a set is determined
(in an easy way) by the trees of its fibers; this is a key reason for
trees of definable sets not being too complicated.

\subsection{Lipschitz continuously varying fibers}

Isomorphisms between the trees $\Tr(X) \to \Tr(X')$ of two sets $X, X' \subset \Zp^n$
correspond to isometric bijections between the $p$-adic closures
$\clX \to \clX'$. More precisely, the following Lemma holds.

\begin{lem}\label{lem:treeIso}
Suppose that $X, X' \subset \Qp^n$ are two arbitrary sets and
$B = \ball{\xtup_0, \lambda}, B' = \ball{\xtup'_0, \lambda'} \subset \Qp^n$
are two balls. Then bijection $\phi\colon B \cap X \to B' \cap X'$ satisfying 
\begin{equation}\label{eq:treeIso}
v(\phi(\xtup_1) - \phi(\xtup_2)) = 
v(\xtup_1 - \xtup_2) - \lambda + \lambda'
\quad \text{for all } \xtup_{1}, \xtup_{2} \in B \cap X
\end{equation}
induces an isomorphism of trees
\[
\begin{aligned}
\phitree\colon \Tr_B(X) &\longrightarrow \Tr_{B'}(X')\\
\ball{\xtup, \mu} &\longmapsto \ball{\phi(\xtup), \mu - \lambda + \lambda'}
,
\end{aligned}
\]
where $\xtup \in B \cap X$ and $\mu \ge \lambda$.
On the other hand, any isomorphism
$\phitree\colon \Tr_B(X) \rightarrow \Tr_{B'}(X')$
induces a bijection $\phi\colon B \cap \clX \to B' \cap \clX'$
satisfying \req{treeIso}.
\end{lem}

\begin{proof}
\req{treeIso} implies that $\phitree$ is well-defined, and an inverse
of $\phi$ induces an inverse of $\phitree$. For the other direction,
note that $B \cap \clX$ is in bijection to the set of infinite paths of
$\Tr_B(X)$ and define $\phi(x)$ as the only element in the intersection
$\bigcap_{\mu \ge \lambda}\phitree(\ball{x,\mu})$.
\end{proof}

A crucial point in the whole analysis of trees is the following
observation:
if $X \subset \Zp \times \Zp$ is a set whose vertical fiber
$X_x$ does not vary too quickly with $x$,
then the tree $\Tr(X)$ is the same as if the fiber would not vary
at all. A similar statement is true in
higher dimensions. We formulate this as a lemma.

\begin{lem}\label{lem:contTree}
Let $X \subset \Zp^m \times \Zp^n$ be any set and denote by $X_{\xtup} := \{\ytup\in
\Zp^n \mid (\xtup,\ytup) \in X \}$
its fiber at $\xtup \in \Zp^m$. Suppose that
for any $\xtup_1, \xtup_2 \in \Zp^m$, any $\ytup \in \Zp^n$ and any
$\lambda \le v(\xtup_1 - \xtup_2)$, we have
$\Tr_{\ytup,\lambda}(X_{\xtup_1}) \cong \Tr_{\ytup,\lambda}(X_{\xtup_2})$.
Then $\Tr(X) \cong \Tr(\Zp^m) \times \Tr(X_{\xtup})$
for any $\xtup \in \Zp^m$.
\end{lem}

\begin{rem}
By rescaling,
the lemma implies a similar statement for a subset $X$ of any ball
$B\subset \Qp^m \times \Qp^n$.
\end{rem}

\begin{proof}
For $\lambda \ge 0$, let
$A_{\lambda} := \{0, 1, \dots, p^\lambda - 1\}^m \subset \Zp^m$
be a set of representatives of the balls of radius $\lambda$, and
define the following ``approximations'' to $X$:
\[
X^{(\lambda)} := \bigcup_{\atup\in A_\lambda}\ball{\atup, \lambda} \times
X_{\atup}
.
\]
In particular $X^{(0)} = \Zp^m \times X_0$. Without loss, we will prove
$\Tr(X) \cong \Tr(X^{(0)})$.
We will verify that
the tree of $X^{(\lambda)}$ coincides with the tree of $X$ up to depth $\lambda$
and define isomorphisms $\psi^{(\lambda)}\colon \Tr(X^{(\lambda)}) \iso \Tr(X^{(\lambda + 1)})$
which are the identity up to depth $\lambda$. By putting these together, we get
an isomorphism $\Tr(X^{(0)}) \iso \Tr(X)$ which is equal to
$\psi^{(\lambda)} \circ \dots \circ \psi^{(0)}$ on nodes of depth less or
equal to $\lambda$.

To check that
$\Tr(X^{(\lambda)})$ and $\Tr(X)$ coincide up to depth $\lambda$, we have to verify
that $X^{(\lambda)} \cap (B \times B') \ne \emptyset$ if and only if
$X \cap (B \times B') \ne \emptyset$
for any ball $B \times B' \subset \Zp^m \times \Zp^n$ of radius $\lambda$.
Fix $\atup \in A_\lambda$ such that $B = \ball{\atup, \lambda}$.
We have $X^{(\lambda)} \cap (B \times B') = B \times (X_{\atup}
\cap B')$, so ``$\Rightarrow$'' is clear.
For ``$\Leftarrow$'', suppose $(\xtup,\ytup) \in X \cap (B \times B')$.
By assumption there exists an isomorphism of trees $\Tr_{B'}(X_{\xtup}) \iso
\Tr_{B'}(X_{\atup})$,
so non-emptiness of $X_{\xtup} \cap B'$ implies non-emptiness of $X_{\atup} \cap B'$.

We define $\psi^{(\lambda)}\colon \Tr(X^{(\lambda)}) \to \Tr(X^{(\lambda + 1)})$
to be the identity up to depth $\lambda$, and it remains to find an isomorphism
$\Tr_{B \times B'}(X^{(\lambda)}) \to \Tr_{B \times B'}(X^{(\lambda + 1)})$
for each ball $B \times B' \subset \Zp^m \times \Zp^n$ of radius $\lambda$.

Set $\{\atup\} := B \cap A_\lambda$ and
$\Atil := B \cap A_{\lambda + 1}$. Then we have
\begin{align*}
X^{(\lambda)} \cap (B \times B') &=
B \times (X_{\atup} \cap B') \qquad \text{and}
\\
X^{(\lambda+1)} \cap (B \times B') &=
\bigcup_{\atiltup \in \Atil}\ball{\atiltup, \lambda+1} \times (X_{\atiltup} \cap B')
.
\end{align*}

By assumption, for each $\atiltup \in \Atil$
we have an isomorphism
$\phi_{\atiltup}\colon \Tr_{B'}(X_{\atup}) \to \Tr_{B'}(X_{\atiltup})$.
Now suppose 
$C \times C' \in \Tr(X^{(\lambda)})$
is a node strictly below $B \times B'$, and let $\atiltup \in \Atil$ be such that
$C \subset \ball{\atiltup, \lambda+1}$.
Then we define $\psi^{(\lambda)}(C \times C') := C \times \phi_{\atiltup}(C')$.
\end{proof}

Combining this lemma with Lemma~\ref{lem:treeIso}, we get:

\begin{cor}\label{cor:contTree}
Let $X \subset \Zp^m\times\Zp^n$ be any set and denote by $X_{\xtup} := \{\ytup\in \Zp^n
\mid (\xtup,\ytup) \in X \}$ its fiber at $\xtup \in \Zp^m$. Suppose that
for any $\xtup_1, \xtup_2 \in \Zp^m$ there exists a bijective isometry
$\phi\colon X_{\xtup_1} \to X_{\xtup_2}$ which additionally satisfies
$v(\phi(\ytup) - \ytup) \ge v(\xtup_2 - \xtup_1)$ for any $\ytup
\in X_{\xtup_1}$. Then
$\Tr(X) \cong \Tr(\Zp^m) \times \Tr(X_{\xtup})$ for any $\xtup \in \Zp^m$.
\end{cor}
\begin{proof}
The condition $v(\phi(\ytup) - \ytup) \ge v(\xtup_2 - \xtup_1)$ ensures that
$\phi$ induces a bijection $\ball{\ytup,\lambda} \cap X_{\xtup_1}
\to \ball{\ytup,\lambda} \cap X_{\xtup_2}$ for any $\ytup\in \Zp^n$
and any $\lambda \le v(\xtup_2 - \xtup_1)$.
Thus Lemma~\ref{lem:treeIso} yields
$\Tr_{\ytup,\lambda}(X_{\xtup_1}) \cong \Tr_{\ytup,\lambda}(X_{\xtup_2})$
and Lemma~\ref{lem:contTree} applies.
\end{proof}

\begin{rem}
Again, a similar statement holds for a subset $X$ of any ball
$B \subset \Qp^m \times \Qp^n$.
\end{rem}

If $X$ satisfies the prerequisites of this corollary, we will say that
\emph{the fiber $X_{\xtup}$ varies Lipschitz continuously with $\xtup$}.

\begin{rem}
An isometry $\psi\colon \Zp^m\times\Zp^n \to \Zp^m\times\Zp^n$ fixing
the first $m$ coordinates preserves Lipschitz continuity of fibers.
\end{rem}

\subsection{Trees of smooth schemes}
\label{subsect:smooth}

We will now prove Theorem~\ref{thm:mainSmooth}
(except for the ``in particular'' part), i.e.\ we will determine the
tree of a scheme which is sufficiently smooth in the sense of the theorem.
Let us first check how to reduce the computation of trees of general schemes
of finite type to trees of affine schemes.

\begin{lem}\label{lem:affin}
Suppose $V$ is a scheme of finite type and $(V_i)_{i \in I}$ is a covering
of $V$. Then for any child $v$ of the root of $\Tr(V)$, there is an $i \in I$
and a child $v'$ of the root of $\Tr(V_i)$ such that the subtree of $\Tr(V)$
starting at $v$ and the subtree of $\Tr(V_i)$ starting at $v'$ are isomorphic.
\end{lem}
\begin{proof}
Denote by $s$ the special point of $\spec \Zp$ and
by $\eta$ the generic one. For some given $\lambda \ge 1$, write
$\sigma\colon\spec \bbF_p \to \spec\bbZ/p^\lambda\bbZ$ and
$\pi\colon\spec \bbZ/p^\lambda\bbZ \to \spec\Zp$ for the canonical maps.

Suppose $v \in V(\bbF_p)$ is a child of the root of $\Tr(V)$.
Choose $i$ such that $V_i$ contains the image of $v$.
The preimage $v'$ of $v$ under the map $V_i(\bbF_p) \to V(\bbF_p)$ is the child
of the root of $\Tr(V_i)$ we are looking for; we have to verify that
the whole tree below $v$ already appears in $\Tr(V_i)$.

Suppose that $w \in V(\bbZ/p^\lambda\bbZ)$ is a node of $\Tr(V)$ below $v$,
i.e.\ $w\circ \sigma = v$, and there exists an $x \in V(\Zp)$ such that
$w = x \circ \pi$. It is clear that $w$ has a preimage
$w' \in V_i(\bbZ/p^\lambda\bbZ)$. As $V_i$ is open and contains
$x(s)$, it also contains $x(\eta)$, so $\im x \subset V_i$.
Thus $x$ has a preimage $x' \in V_i(\Zp)$, and 
$w' = x' \circ \pi$.
\end{proof}

\begin{proof}[Proof of Theorem~\ref{thm:mainSmooth}]
Let $V$ be a scheme as in the theorem.
By Lemma~\ref{lem:affin}, it suffices to consider affine $V$;
we fix an embedding $V \inject \bbA^n$ and determine the tree of
$V(\Qp) \subset \Qp^n$.

Fix $\ztup \in V(\Qp) \cap \Zp^n$, and suppose that the dimension of
$V$ at $\ztup$ is $d$. We determine the tree
on a small ball $B := \ball{\ztup, \lambda}$ around $\ztup$.
Write $B$ as a product $\BX \times \BY$, with $\BX \subset \Zp^{d}$
and  $\BY \subset \Zp^{n - d}$, and denote the coordinates
by $X_1, \dots, X_d, Y_1, \dots , Y_{n-d}$.
To simplify notation, suppose $\ztup = 0$.

Let $f_1, \dots, f_{n-d} \in \Zp[X_1, \dots, X_d, Y_1, \dots , Y_{n-d}]$
be generators of the ideal
of $V$ in the local ring at $0$; regularity of that ring implies
that indeed $n-d$ polynomials suffice. Moreover, after possibly permuting
coordinates, the matrix $(\frac{\del f_i}{\del Y_j}(0))_{1\le i,j \le n-d}$ is
invertible over $\Qp$.
$\GL_n(\Zp)$ acts on $B$ by isometries, so by Lemma~\ref{lem:treeIso},
applying such matrices does not change the tree of $V(\Zp)$ on $B$.
Thus by using the column transformations of the smith normal form,
we may additionally suppose that
$\frac{\del f_i}{\del X_j}(0) = 0$  for $i \le n-d, j \le d$.

Now we apply the implicit function theorem (see e.g.~\cite{Igu:locZeta}).
This yields a power series
$\atup$ with coefficients in $\Qp$, from the variables $X_i$
to the variables $Y_j$
such that for $\lambda \gg 0$,
$\atup$ converges on $\BX$, and
for $(\xtup, \ytup) := (x_1, \dots, x_d, y_1, \dots, y_{n-d}) \in B$, we have
$(\xtup, \ytup) \in V(\Qp)$ if and only if $\ytup = \atup(\xtup)$.
As $\frac{\del f_i}{\del X_j}(0) = 0$,
this power series has no linear term,
so for $\lambda$ sufficiently large and $\xtup, \xtup' \in \BX$, we get
\begin{equation}\label{eq:smoothCont}
v(\atup(\xtup) - \atup(\xtup'))
\ge v(\xtup - \xtup')
;
\end{equation}
in particular, $\atup(\xtup) \in \BY$ for $\xtup \in \BX$.
Thus the fiber of $V(\Qp) \cap B$ at $\xtup \in \BX$ is exactly
$\{\atup(\xtup)\}$, and by \req{smoothCont},
it varies Lipschitz continuously with $\xtup$; hence Corollary~\ref{cor:contTree}
yields $\Tr_B(V(\Qp)) \cong \Tr(\Zp^d)$.

As $V(\Qp) \cap \Zp^n$ is compact in $p$-adic topology, we can cover it
by finitely many balls $B$ satisfying
$\Tr_B(V(\Qp)) \cong \Tr(\Zp^d)$ (possibly for different $d$, but
all satisfying $d \le \dim V$).
Moreover, in $\bbZ^n_p$ any two balls are either disjoint or contained in one
another, so
we may suppose that these balls $B$ are all disjoint. Thus the total
tree of $V(\Qp)$ consists of a finite tree (the subtree of $\Tr(\Zp^n)$
whose leaves are exactly the balls used in the cover), with a copy of
$\Tr(\Zp^d)$ attached to each leaf.
\end{proof}

The ``in particular'' part of Theorem~\ref{thm:mainSmooth} will be a direct
consequence of Lemma~\ref{lem:firstProp}.

\subsection{Example: the cusp curve}
\label{subsect:cusp}

Up to now, we only saw very simple trees. As a more complicated example, 
let us compute the tree of the
cusp curve $X = \{(x, y) \in \Zp^2 \mid x^3 = y^2\}$ when $p \ne 2$.
This tree will already contain most of the aspects appearing
in the general definition of level $d$ trees.

We will need the following notation: let $\bifurc(\kappa)$
be the tree which starts with a path of length $\kappa$ and then
has a bifurcation into two infinite paths; in other words,
$\bifurc(\kappa)$ is the tree of a two-point-set $\{x_1, x_2\}$,
where $v(x_1 - x_2) = \kappa$.

From the previous subsection, it is clear that $\Tr(X)$ might
be complicated only close to $(0,0)$; thus we will determine the tree
on squares which do not contain $(0, 0)$ and then put them
together. The largest squares not containing $(0,0)$ are of the form
$B = \ball{(x_0,y_0), \kappa + 1}$ with $\kappa = v(x_0,y_0)$.
Fix such $x_0,y_0,\kappa$.

If $v(x_0) > v(y_0)$, then $v(x) > v(y)$ for any $(x,y) \in B$.
This implies $x^3 \ne y^2$, so $B \cap X$ is empty. Thus in the following
we suppose $\kappa = v(x_0) \le v(y_0)$.

Write $B$ as a product
$\BX \times \BY = \ball{x_0, \kappa + 1} \times \ball{y_0, \kappa + 1}$,
and let us analyse the fiber of $X$ at some $x \in \BX$.
It is $X_x = \{\pm \sqrt{x^3}\}$
if this root exists and empty otherwise. By Hensels Lemma,
the root $\sqrt{x^3} = x\sqrt{x}$
exists if and only if $v(x)$ is even and the angular component of
$x$ is a square in the residue field $\bbF_p$.
Neither $v(x)$ nor the angular component of $x$ depend on
the specific choice of $x \in \BX$, so either all $X_x$ are empty
or all $X_x$ consist of two roots (for $\BX$ fixed).

If the roots don't exist, then $B \cap X$ is empty, so suppose
now that they do exist. Consider two elements $x_1, x_2 \in \BX$.
By applying Lemma~\ref{lem:root} to $\sqrt{\frac{x_1}{x_2}}$, one
checks that there is a suitable choice of roots
$\sqrt{x_1^3}$, $\sqrt{x_2^3}$ such that
\begin{equation}\label{eq:cuspFlat}
v\!\left(\textstyle{\sqrt{x_1^3} - \sqrt{x_2^3}}\right) \ge v(x_1 - x_2)
.
\end{equation}
In particular, $\sqrt{x_1^3} \in \BY$ if and only if
$\sqrt{x_2^3} \in \BY$. Moreover $v\big(\sqrt{x^3} - (-\sqrt{x^3})\big)$ does
not depend on $x \in \BX$, so we may apply
Corollary~\ref{cor:contTree} and get
$\Tr_{B}(X) \cong \Tr(\Zp) \times \Tr_{\BY}(X_x)$ for any $x \in \BX$.
It remains to determine
$\Tr_{\BY}(X_x)$. We have $v\big(\sqrt{x^3}\big) = v\big(\sqrt{x^3} - (-\sqrt{x^3})\big) =
\frac32\kappa$, so we get: if $\kappa = 0$, then there exist two balls
$\BY$ such that $\Tr_{\BY}(X_x) = \Tr(\Pt)$, and all other $\BY \cap X_x$
are empty; if $\kappa > 0$, then
$\Tr_{0, \kappa+1}(X_x) \cong \bifurc(\frac12\kappa - 1)$, and all other
$\BY \cap X_x$ are empty.

\begin{figure}
\newcommand{\kr}{[fill] circle (1pt)}
\newcommand{\duenn}{[line width=.3pt]}
\newcommand{\dick}{[line width=1pt]}
\newcommand{\etchoch}{node[anchor=south] {\vdots}}
\begin{tikzpicture}[x=6mm,y=5mm]
  
\draw\duenn
  (0,0) \kr -- (-1,1) \kr -- (-2,2) \kr -- (-3,3) \kr -- (-4,4) \kr --
  (-5,5) \kr -- (-5.7,5.7) node[anchor=south east] {$\ddots$};

\node at (-5.5,8) {$\ddots$};

\foreach \x in {-0.25,0.25,...,1.25} {
  \draw\duenn (0,0) -- (\x,1) \kr;
  \draw\dick (\x,1) -- ++(0,1) \kr -- ++(0,.7) \etchoch;
}

\foreach \x in {-2,-1} {
  \draw\duenn (-2,2) -- (\x,3) \kr;
  \draw\dick (\x,3) -- ++(-0.25,1) \kr -- ++(0,.7) \etchoch
             (\x,3) -- ++(+0.25,1) \kr -- ++(0,.7) \etchoch;
}

\foreach \x in {-4,-3} {
  \draw\duenn (-4,4) -- (\x,5) \kr;
  \draw\dick (\x,5) -- (\x,6) \kr -- ++(-0.25,1) \kr -- ++(0,.7) \etchoch
                       (\x,6) \kr -- ++(0.25,1) \kr -- ++(0,.7) \etchoch;
}

\end{tikzpicture}
\caption{The tree of the cusp curve $X = \{(x, y) \in \bbZ_5^2 \mid x^3 = y^2\}$; thick lines
mean ``multiply by $p$''.}
\label{fig:cusp}
\end{figure}
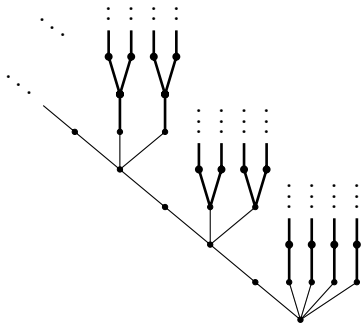

Assembling our results, we get the total tree of $X$ (see
Figure~\ref{fig:cusp}):
it consists of an infinite path
(the nodes $\ball{0, \kappa}$ for $\kappa \ge 0$) with some side branch attached to it.
The root has $p-1$ additional children, and each of these children is the
root of a copy of $\Tr(\Zp)$.
(The number $p-1$ comes from the fact that $\bbF_p$ contains
$\frac{p-1}{2}$ squares and each such square contributes two children.)
Finally, for each $\kappa \in 2\Gamma$, $\kappa \ge 2$,
the node $\ball{0, \kappa}$ has
$\frac{p-1}{2}$ additional children, each of which is the root
of a copy of $\Tr(\Zp) \times \bifurc(\frac12\kappa - 1)$.

\section{Trees of level $d$}
\label{sect:defTree}

\subsection{Definition of trees of level $d$}
\label{subsect:defTree}

We will now define a \emph{tree datum for a tree of level $d$} and
explain how to construct an actual tree out of it. By definition,
an arbitrary tree is of level $d$ if it is isomorphic to a tree
constructed in this way.
Such a tree of level $d$ will consist of a ``skeleton'' which has
only finitely many bifurcations, together with
trees of level $d-1$ attached to every node in some uniform way.
For this uniformity to make sense, we need a parametrized version of
these notions. A \emph{parametrized tree} is a map $\tree\colon M \to \Trees$,
where $M \subset \Gamma^m$ is definable.

A \emph{datum for a level $0$ tree} defined on $M \subset \Gamma^m$ consists of:
\begin{itemize}
\item
a finite tree $\skel$ (possibly empty)
\item
for each edge $\bone = (\joint, \joint')$ of $\skel$
a linear function $\ell_{\bone}\colon M\to \Gamma_{>0} \cup \{\infty\}$
(the ``length of $\bone$''); the value $\infty$ is only allowed if $\joint'$
is a leaf of $\skel$.
\end{itemize}
The nodes of $\skel$
will be called \emph{joints}; the edges will be called \emph{bones}.
An \emph{virtual joint} is a leaf following a bone of infinite length;
the other joints are \emph{real joints}.

Out of such a datum one constructs a tree $\tree(\kappatup)$
(for $\kappatup \in M$) as follows.
Start with a copy of $\skel$, but omitting the virtual joints,
and denote the copy of the joint $\joint \in \skel$
by $\joint(\kappatup)$. For each bone
$\bone = (\joint, \joint')$, add $\ell_{\bone}(\kappatup) - 1$
nodes between $\joint(\kappatup)$ and $\joint'(\kappatup)$ if $\joint'$
is real (thus creating a path of length $\ell_{\bone}(\kappatup)$ from
$\joint(\kappatup)$ to $\joint'(\kappatup)$), and add 
an infinite path below $\joint(\kappatup)$ if $\joint'$ is virtual;
denote the set of these new nodes by $\bone(\kappatup)$.

The \emph{depth} $\depth(\joint)$ of a joint is the function $\kappatup \mapsto
\depth(\joint(\kappatup))$ if $\joint$ is real and $\kappatup \mapsto \infty$
if $\joint$ is virtual.

Note that a given level $0$ tree $\tree\colon M \to \Trees$
can be described by a tree datum in different ways.
In particular, we may replace a bone of $\tree$
by several bones (of appropriate lengths) with joints in between.

\medskip

Before we describe level $d+1$ trees, we need to describe how
side branches of such trees look like. A \emph{datum for a side branch of
level $d$} (defined on $M$) consists of:
\begin{itemize}
\item
a non-empty finite tree $\fintree$
\item
for each leaf $w$ of $\fintree$, a level $d$ tree
$\tree_w\colon M \to \Trees$ such that $\tree_w(\kappatup)$ is non-empty.
\end{itemize}
The corresponding side branch $\branch(\kappatup) \in \Trees$
(for $\kappatup \in M$) consists of
$\fintree$ with $\Tr(\Zp) \times \tree_w(\kappatup)$ attached to
$w$ for each leaf $w$ of $\fintree$.

\medskip

Now, a \emph{datum for a level $d+1$ tree} (defined on $M$) is the following:
\begin{itemize}
\item
an element $\rho \in \Gamma_{>0}$
\item
a datum $(\skel, (\ell_{\bone}))$ for a level $0$ tree (defined on $M$),
such that for any bone $\bone$, the length
$\ell_{\bone}(\kappatup) \mod \rho$ does not depend on $\kappatup$;
denote by $\tree_0$ the tree build out of $(\skel, (\ell_{\bone}))$
\item
for each real joint $\joint$ of $\tree_0$, a side branch datum
defining a side branch $\branch_{\joint}\colon M \to \Trees$ of level $d$
\item
for each bone $\bone = (\joint, \joint')$ of $\tree_0$ and each congruence class
$\ccGamma \in \Gamma/\rho\Gamma$,
a side branch datum defining a side branch
$\branch_{\bone,\ccGamma}\colon N_{\bone,\ccGamma} \to \Trees$ of level $d$,
where
\[
N_{\bone,\ccGamma} = \{(\kappatup, \lambda) \in M \times \ccGamma
\mid \depth(\joint)(\kappatup)< \lambda < \depth(\joint')(\kappatup)\}
.
\]
\end{itemize}
The tree $\tree(\kappatup)$ is constructed as follows.
Start with $\tree_0(\kappatup)$, and to each node $v \in \tree_0(\kappatup)$
attach a side branch:
if $v = \joint(\kappatup)$ for some joint $\joint$, then attach
$\branch_{\joint}(\kappatup)$ to $v$.
Otherwise $v \in \bone(\kappatup)$ for some bone $\bone$, and
$\depth(v) \in \ccGamma$ for some $\ccGamma \in \Gamma/\rho\Gamma$.
Attach $\branch_{\bone,\ccGamma}(\kappatup, \depth(v))$ to $v$.

$\tree_0$ will be called the \emph{skeleton}
of $\tree$, and the joints and bones of $\tree$ are the joints and bones
of $\tree_0$.
The trees of level $d$ appearing in the side branch data 
will be called the \emph{side trees} of $\tree$. (Note that it does
not make sense to say that a side
tree is a subtree: some side trees are not even parametrized by the same set.)

An \emph{unparametrized tree of level $d$} is a parametrized tree
of level $d$ defined on the one-point set $M = \Gamma^0$.

\subsection{Piecewise level $d$ trees}
\label{subsect:defPiece}

In the definition of the previous subsection, we tried to be as restrictive
as possible. We will now show how one can weaken the conditions on
parametrized level $d$ trees without changing the notion of
unparametrized trees. While our first definition is useful to deduce
other statements about trees, the new definition will be more convenient
to show that a given tree is of level $d$.

Define a \emph{generalized level $d$ tree} in the same way as an
ordinary one, with the following modifications:
given a bone $\bone = (\joint, \joint')$, instead of cutting
\begin{equation}\label{eq:defNe}
N_{\bone} := \{(\kappatup, \lambda) \in M \times \Gamma
\mid \depth(\joint)(\kappatup)< \lambda < \depth(\joint')(\kappatup)\}
\end{equation}
into subsets according to $\lambda \mod \rho$, we allow $N_{\bone}$
to be cut into finitely many arbitrary definable subsets $N_{i}$
and use a separate side branch datum $S_{\bone,i}$ for each such subset.
Moreover, the condition on the length of
the bones modulo $\rho$ is removed, and
the side trees of a generalized level $d$ tree
are also allowed to be generalized. 

\begin{lem}\label{lem:genTree}
Unparametrized generalized level $d$ trees
are the same as unparametrized normal level $d$ trees.
\end{lem}

In the proof of this lemma, we will use trees $\tree\colon M \to \Trees$ which are only
\emph{piecewise of level $d$} (normal or generalized):
there exists
a finite partition of $M$ into definable subsets $M_{i}$, such that
each restricted tree $\tree\auf{M_{i}}$ is of level $d$ (normal
or generalized).
As ``piecewise'' only concerns parameters, Lemma~\ref{lem:genTree} is
a special case of the following lemma.

\begin{lem}\label{lem:genTreeParam}
Piecewise generalized level $d$ trees
are the same as piecewise normal level $d$ trees.
\end{lem}

\begin{proof}[Proof of Lemma~\ref{lem:genTreeParam}]
We use induction over the level. For $d = 0$, the statement is trivial.

Suppose now $\tree$ is piecewise a generalized level $d\ge 1$ tree.
We have to show that $\tree$ is also piecewise a normal level $d$ tree.
It is clear that for generalized trees, it does not make any
difference whether we allow
the side trees to be piecewise or not, so using the
induction hypothesis, we may suppose the side trees of $\tree$ to be ungeneralized of
level $d - 1$.

Now consider a bone $\bone$ of $\tree$
and the corresponding decomposition
of the set $N_{\bone}$ into definable subsets $N_i$ (defined in \req{defNe} above).
When working with ungeneralized trees, we are a priori only allowed to decompose
$N_{\bone}$ into sets of the form $N_{\bone} \cap (M \times \ccGamma)$
for $\ccGamma \in \Gamma/\rho\Gamma$.
But modifications of the tree
also permit us to do some other cuts: as we are working with piecewise
trees, we may intersect $N_{\bone}$ with sets of the form $M' \times \Gamma$ for
$M' \subset M$ definable, and moreover, we may 
cut the bone $\bone$ into several bones, thus intersecting $N_{\bone}$
with sets of the form $\{(\kappatup, \lambda) \mid \lambda \lesseqqgtr
\ell(\kappatup)\}$.
By Lemma~\ref{lem:cellGamma} any definable subset of $N_{\bone}$ may
be obtained in this way, if arbitrary $\rho$ are allowed.

It remains to deal with the requirement to have
one single $\rho$ for the whole tree, and that the
lengths of the bones have to be constant modulo $\rho$. But we may use the
least common multiple of all $\rho$ we need and then further cut $M$
into definable subsets according to the congruence classes of the
lengths of bones.
\end{proof}

In this subsection, we introduced a lot of different kinds of
trees of level $d$. In the remainder of the article, we will
only use normal and generalized piecewise ones.
Having Lemma~\ref{lem:genTreeParam} in mind, 
generalized piecewise trees will be just called piecewise trees.

\subsection{First properties of level $d$ trees}
\label{subsect:firstProp}

To familiarize with level $d$ trees, let us verify the
following simple lemmas.
\begin{lem}\label{lem:firstProp}
\begin{enumerate}
\item\label{it:fp0}
An unparametrized level $0$ tree consists of a finite tree
with finitely many infinite paths attached to it.
\item\label{it:m+1}
Any (piecewise or not) level $d$ tree is also a
(piecewise or not) level $d+1$ tree.
\item\label{it:Zpn}
If $\tree$ is a level $d$ tree, then $\Tr(\Zp) \times \tree$
is a level $d+1$ tree. In particular, $\Tr(\Zp^n)$ is a level
$n$ tree.
\item\label{it:glue}
Suppose that $\tree_1, \tree_2\colon M \to \Trees$ are parametrized trees
defined on the same set,
that $\tree_1$ is of level $d$ and that
$\tree_2$ is piecewise of level $d$. Suppose moreover that
$\joint$ is a real joint of $\tree_1$ and that $\tree_2(\kappatup) \ne \emptyset$
for any $\kappatup \in M$. Let $\tree(\kappatup)$ be the tree one gets by
attaching $\tree_2(\kappatup)$ to $\tree_1(\kappatup)$ at $\joint(\kappatup)$.
Then $\tree$ is piecewise of level $d$.
\end{enumerate}
\end{lem}
\begin{proof}
(1)
Clear.

(2)
By induction, it is enough to verify this for $d = 0$. A level $0$ tree
is a level $1$ tree with side branches consisting only of the root.

(3)
Let the skeleton of $\Tr(\Zp) \times \tree$ consist only of the root,
let the finite tree $\fintree$ in the side branch at the root also consist
only of the root, and attach $\Tr(\Zp) \times \tree$ to the only leaf of
$\fintree$.

(4)
Clear (using generalized level $d$ trees).
\end{proof}

\begin{lem}\label{lem:subTree}
Let $\tree$ be an unparametrized tree of level $d$
and let $v$ be any node of $\tree$. Then the subtree of $\tree$
below $v$ is of level $d$.
\end{lem}
\begin{proof}
If $v$ lies on the skeleton or on the
finite tree at the beginning of a side branch, then this is easy.
If $v$ lies in $\Tr(\Zp) \times \tree'(\lambda)$
for some side tree $\tree'$ and some
$\lambda \in \Gamma$, then $\tree'(\lambda)$ is of level
$d - 1$ as an unparametrized tree. By induction,
the subtree of $\tree'(\lambda)$ starting at the image of $v$
is of level $d - 1$, hence the subtree starting at $v$ is of level
$d$ by Lemma~\ref{lem:firstProp} \rit{Zpn}.
\end{proof}

It is now easy to see that it suffices
to understand trees of affine schemes to get trees of arbitrary schemes.

\begin{prop}\label{prop:affin}
Let $V$ be an arbitrary scheme of finite type, and suppose that $V$
has an affine
covering $(V_i)_{i \in I}$ such that each $\Tr(V_i)$ is of level $d$.
Then $\Tr(V)$ is of level $d$.
\end{prop}
\begin{proof}
Use Lemma~\ref{lem:affin}, Lemma~\ref{lem:subTree}
(applied to the children of the roots of the trees $\Tr(X_i)$)
and Lemma~\ref{lem:firstProp} \rit{glue}.
\end{proof}

We conclude this subsection by a lemma which enables us to
decompose the computation of a tree into separate computations
on a cheese and its holes.

\begin{lem}\label{lem:glueCheese}
Suppose we have, for each $\kappatup$ in some definable set
$M \subset \Gamma^{m}$,
a set $X_{\kappatup} \subset \Zp^n$ and a cheese 
$S_{\kappatup} := \Zp^n \ohne \bigcup_{i \in I} B_{\kappatup,i}$,
where the index set $I$ does not depend on $\kappatup$.
Suppose moreover that the following holds:
\begin{enumerate}
\item
$\kappatup \mapsto \Tr_{S_{\kappatup}}(X_{\kappatup})$ is of level
$d$.
\item
For each $i \in I$, $\kappatup \mapsto \Tr_{B_{\kappatup,i}}(X_{\kappatup})$
is piecewise of level $d$.
\item
For each $i \in I$,
there is a joint $\joint_i$ of $\kappatup \mapsto \Tr_{S_{\kappatup}}(X_{\kappatup})$
such that $\joint_i(\kappatup) = B_{\kappatup,i}$ for all $\kappatup \in M$.
\end{enumerate}
Then the whole tree $\kappatup \mapsto \Tr(X_{\kappatup})$
is piecewise of level $d$.
\end{lem}

\begin{proof}
The third condition in particular implies $X_{\kappatup} \cap B_{\kappatup,i}
\ne \emptyset$, so
$\Tr(X_{\kappatup})$ consists of
$\Tr_{S_{\kappatup}}(X_{\kappatup})$ with $\Tr_{B_{\kappatup,i}}(X_{\kappatup})$ attached to
it at the node $B_{\kappatup,i}$ for each $i \in I$. Now use
Lemma~\ref{lem:firstProp} \rit{glue}.
\end{proof}

\section{Results on trees of level $d$}
\label{sect:results}

\subsection{Rationality of the Poincar\'e series}
\label{subsect:poincare}

In the introduction we promised that level $d$ trees would
have rational Poincar\'e series. Let us now make this precise
and verify it.

\begin{defn}
Suppose $\tree$ is a tree which has only finitely many
nodes at each depth. Then we define the \emph{Poincar\'e series of $\tree$} as follows:
\[
P_{\tree}(\sv) := \sum_{\lambda = 0}^\infty \#\{v \in \tree \mid \depth(v) = \lambda\} \cdot \sv^\lambda
\in \bbZ[[\sv]].
\]
\end{defn}

\begin{prop}\label{prop:rat}
Let $\tree$ be a level $d$ tree. Then $P_{\tree}(\sv) \in \bbQ(\sv)$.
\end{prop}

The main ingredient to the proof of this proposition is the following lemma:
\begin{lem}\label{lem:rat}
Suppose $M \subset \Gamma^{m}$ is a definable set contained in
$\Gamma_{\ge 0}^{m}$. Then the series
\[
\sum_{(\kappa_1, \dots, \kappa_m) \in M}\!\!\!\!\!\!\!\!
Y_1^{\kappa_1} \cdots Y_m^{\kappa_m}
\in \bbZ[[Y_1, \dots, Y_m]]
\]
is rational in $Y_1, \dots, Y_m$.
\end{lem}
This is, for example, a simplified version of
Theorem~4.4.1 of \cite{CL:mot}.
\begin{proof}[Sketch of proof]
Using cell decomposition in $\Gamma^{m}$ and by further refining the cells,
one reduces the statement to sums of the form
\[
\sum_{\kappa_1 = 0}^{\beta_1}
\sum_{\kappa_2 = 0}^{\beta_2(\kappa_1)}
\dots
\sum_{\kappa_m = 0}^{\beta_m(\kappa_1,\dots,\kappa_{m-1})}
Y_1^{\ell_1(\kappa_1)}
\dots
Y_m^{\ell_m(\kappa_m)}
\]
where the $\ell_i$ are linear and non-constant, the $\beta_i$
are linear or $\infty$, and $\beta_i(\kappa_1, \dots, \kappa_{i-1}) \ge 0$
for all appearing tuples $(\kappa_1, \dots, \kappa_{i-1})$.
Now use inductively that geometric series are rational.
\end{proof}

\begin{proof}[Proof of Proposition~\ref{prop:rat}]
We inductively prove the following parametrized version of the proposition.
Let $M \subset \Gamma_{\ge 0}^{m}$ be a definable set and
let $\tree \colon M \to \Trees$ be a parametrized level $d$ tree.
Then the series
\begin{equation}\label{eq:paramRat}
P_{\tree}(\sv, Y_1, \dots, Y_m) :=
\!\!\!\!\!\!\!\!\sum_{(\kappa_1, \dots, \kappa_m) \in M}\!\!\!\!\!\!\!\!
P_{\tree(\kappatup)}(\sv)\cdot
Y_1^{\kappa_1} \cdots Y_m^{\kappa_m}
\in \bbZ[[\sv, Y_1, \dots, Y_m]]
\end{equation}
is rational in $\sv, Y_1, \dots, Y_m$.
Note that the condition $M \subset \Gamma_{\ge 0}^{m}$ is satisfied for
iterated side trees of level $d$ trees.

If we define a level $-1$ tree to be one consisting only of the root,
then we may view a level $0$ tree as one having side branches
of level $-1$ (and where additionally the finite trees $\fintree$
at the beginning of the side branches consist only of the root). Adopting this point
of view, we start our induction at $d = -1$.

If $d= -1$, then $P_{\tree(\kappatup)}(\sv) = 1$ for all $\kappatup \in M$,
and Equation~\req{paramRat} is just Lemma~\ref{lem:rat}.

If $\tree'(\kappatup) \cong \Tr(\Zp) \times \tree(\kappatup)$, then
$P_{\tree'}(\sv, Y_1, \dots, Y_m) = P_{\tree}(p\sv, Y_1, \dots, Y_m)$.
Using this, rationality of level $d$ trees implies rationality
of level $d$ side branches.

Now consider a level $d+1$ tree $\tree$ defined on
$M \subset \Gamma_{\ge 0}^{m - 1}$.
We may treat each joint and each bone separately. Moreover,
on each bone we may treat the different congruence classes modulo $\rho$
separately. The total Poincar\'e series $P_{\tree}(\sv, Y_1, \dots, Y_{m-1})$
is then the sum of all these parts.

Consider a bone $\bone = (\joint, \joint')$ and a congruence class
$\ccGamma \in \Gamma/\rho\Gamma$. Let $\branch$ be the tree in $m$
parameters describing the side branches at nodes on $\bone$ with depth in
$\ccGamma$. The contribution of these side branches, including the corresponding
nodes on $\bone$ themselves, is
$P_{\branch}(\sv, Y_1, \dots, Y_{m - 1}, \sv)$.

Finally consider a (real) joint $\joint$ with side branch $\branch$.
We
define $M' := \{(\kappatup,\depth(\joint)(\kappatup)) \mid \kappatup \in M)\}$
and apply the induction hypothesis to the ``shifted'' tree
\[
\branch'\colon M' \to \Trees, (\kappatup, \lambda) \mapsto \branch(\kappatup)
.
\]
The contribution of $\joint$ and its side branch
is $P_{\branch'}(\sv, Y_1, \dots, Y_{m - 1}, \sv)$.
\end{proof}

\subsection{Any level $d$ tree appears}

We now prove Theorem~\ref{thm:inv}: any tree of level $d$ without leaves
is isomorphic to a tree of a definable set of dimension $d$.
We introduce some additional notation only for this subsection.
The coordinates of any $m$-tuple $\atup$ will be denoted by $a_1, \dots, a_m$.
Moreover, for $\xtup \in \Qp^m$ we will set
$\vtup(\xtup) := (v(x_1), \dots, v(x_m))$
(in contrast to $v(\xtup) = \min_i v(x_i)$).

The main ingredient to the proof is the following lemma.

\begin{lem}\label{lem:invFkt}
Suppose $M \subset \Gamma_{\ge0}^m$ is definable and 
$\ell\colon M \to \Gamma_{\ge0}$ is a linear function satisfying 
$\ell(\kappatup) \ge \kappa_i$ for each $i \le m$.
Define $X := \{\xtup \in \Zp^m \mid \vtup(\xtup) \in M\}$.
Then there exists a definable function
$\pow_\ifu\colon X \to \Zp$ with the following properties:
\begin{enumerate}
\item
$v(\pow_\ifu(\xtup)) = \ell(\vtup(\xtup))$ for any $\xtup \in X$, and
\item
$v(\pow_\ifu(\xtup)-\pow_\ell(\xtup')) \ge v(\xtup - \xtup')$
for any $\xtup, \xtup' \in X$ satisfying $\vtup(\xtup) = \vtup(\xtup')$.
\end{enumerate}
\end{lem}
\begin{proof}
Write $\ell(\kappatup) =: \frac{1}{e}(\beta + \sum_i a_i\kappa_i)$ with $a_i \in
\bbZ$, $\beta \in \Gamma$, $e \in \bbN_{>0}$.
Set $\mu := 1 + 2v(e)$. For $x \in G := p^{e\Gamma}\cdot\ball{1,\mu}$, write
$\sqrt[e]{x}$ for the $e$-th root of $x$ lying in $p^\Gamma\cdot\ball{1,1+v(e)}$
(which exists by Lemma~\ref{lem:root}).
Choose representatives $r_\nu \in \Zp\mult$ of the sets $\Zp\mult/\ball{1,\mu}$.
Using these choices, we define $\pow_{\ifu}$ as follows.

First suppose $1 \le i \le m$ and $0 \le \lambda < v(e)$,
and consider the definable set $X_{i,\lambda} := \{\xtup \in X \mid
\ifu(\vtup(\xtup)) = v(x_i) + \lambda\}$. For $\xtup \in X_{i,\lambda}$,
we define $\pow_{\ifu}(\xtup) := p^\lambda x_i$. This satisfies
both required conditions, so we may remove $X_{i,\lambda}$ from $X$.
We do this successively for all $i \le m$ and all $\lambda < v(e)$
and henceforth suppose that
\begin{equation}\label{eq:lvGross}
\ifu(\vtup(\xtup)) \ge v(x_i) + v(e)
\end{equation}
for $\xtup \in X$ and all $i$.

For $\xtup \in X$, set
$\pi(\xtup) := p^\beta\prod_{i = 1}^m x_i^{a_i}$. As $\ifu$ is defined
on $\vtup(\xtup)$, we get $v(\pi(\xtup)) = e\cdot\ifu(\vtup(\xtup)) \in e\Gamma$,
so $\pi(\xtup)$ lies in $p^{e\Gamma}\ball{1,\mu}r_\nu$
for some $\nu$. Thus $\frac{\pi(\xtup)}{r_\nu} \in G$, and we define
$\pow_{\ifu}(\xtup) := \sqrt[e]{\frac{\pi(\xtup)}{r_\nu}}$.

It is clear from the definition that
$v(\pow_\ifu(\xtup)) = \ell(\vtup(\xtup))$.
Now suppose we have $\xtup, \xtup' \in X$ with $\vtup(\xtup) = \vtup(\xtup')$.
As both $\pow_\ifu(\xtup)$ and $\pow_\ifu(\xtup')$ lie in
$p^{\ell(\vtup(\xtup))}\ball{1,1 + v(e)}$, we have
$v(\pow_\ifu(\xtup) - \pow_\ifu(\xtup')) \ge \ell(\vtup(\xtup)) + 1 + v(e)$;
so the second condition is satisfied unless
\begin{equation}\label{eq:diffGross}
v(\xtup - \xtup') > \ifu(\vtup(\xtup)) + 1 + v(e)
.
\end{equation}

Set $\delta := v(\xtup - \xtup') - \max\{v(x_i) \mid 1 \le i \le m\}$.
By \req{lvGross} and \req{diffGross}, we have $\delta > \mu$ and
in particular $\delta > 0$. By definition $\delta \le v(x_i - x_i') - v(x_i)$
for all $i$, so we have $x_i \etwa x_i'$,
which implies
$\pi(\xtup) \etwa[\delta] \pi(\xtup')$.
As $\delta > \mu$, we have
$\pow_{\ifu}(\xtup) = \sqrt[e]{\frac{\pi(\xtup)}{r_\nu}}$ and
$\pow_{\ifu}(\xtup') = \sqrt[e]{\frac{\pi(\xtup')}{r_\nu}}$
for the same $r_\nu$, so Lemma~\ref{lem:root} yields
$\pow_\ifu(\xtup) \etwa[\delta - v(e)] \pow_\ifu(\xtup')$;
hence
$v(\pow_\ifu(\xtup) - \pow_\ifu(\xtup')) \ge
v(\pow_\ifu(\xtup)) + \delta - v(e) \ge v(\xtup - \xtup')$ by \req{lvGross}.
\end{proof}

In the main proof, we will use the following ``Lipschitz union argument'' several
times: we will have two (or more) sets $X, X' \subset \Zp^m \times \Zp^N$
with Lipschitz continuous fibers in the first $m$ variables and would like to
infer that the union has Lipschitz continuous fibers, too.
This is possible if for any $\xtup_{1}, \xtup_{2} \in \Zp^m$, the corresponding
isometries $\phi\colon X_{\xtup_1} \to X_{\xtup_2}$ and $\phi'\colon X'_{\xtup_1} \to
X'_{\xtup_2}$ satisfy $v(\phi(\ytup) - \phi'(\ytup')) = v(\ytup - \ytup')$
for $\ytup \in X_{\xtup_1}, \ytup' \in X_{\xtup'_1}$.
In particular, this is true if $v(\ytup - \ytup')$
does not depend at all on
$\xtup \in \Zp^m, \ytup \in X_{\xtup}, \ytup' \in X'_{\xtup}$.

\begin{proof}[Proof of Theorem~\ref{thm:inv}]
$\kappatup$ and $\mutup$ will denote elements of $\Gamma^{m}$.
It will be useful to define $\kappa_0 := \mu_0 := 0$.
We will work inside $\Zp^{m+N}$ for some large $N$; $(\xtup, \ytup)$
will be an element of $\Zp^{m+N}$, where $\xtup \in \Zp^m$ and 
$\ytup \in \Zp^N$. Sometimes, we will also write $\ytup = (z, \ystup)$,
with $z \in \Zp$ and $\ystup \in \Zp^{N-1}$.
We will denote the fiber of a set $X \subset \Zp^{m+N}$
at $\xtup \in \Zp^m$ by $X_{\xtup}$.

Let us formulate a suitable parametrized version of the statement, which
we will then prove by induction over the level of the tree.
We start with the following data:
a definable set $M \subset \Gamma^{m}$, a tree $\tree\colon M \to \Trees$
of level $d$ without leaves, and a tuple $\mutup \in \Gamma_{>0}^{m}$.
We suppose that for any $\kappatup \in M$,
we have $\kappa_{i-1} + \mu_{i-1} \le \kappa_{i}$ for $i \in \{1, \dots, m\}$
(i.e.\ $M$ is contained in an ``upper triangle'').

Using this, we define a set $G \subset \Zp^{m}$ as follows.
For $\kappatup \in M$, define the rectangle
\[
G_{\kappatup} :=
p^{\kappa_1}\ball{1, \mu_1} \times \dots \times
p^{\kappa_m}\ball{1, \mu_m}
,
\]
and set $G := \bigcup_{\kappatup \in M}G_{\kappatup}$.
It will also be useful to define $\lambda(\kappatup) := \kappa_{m} + \mu_{m}$
for $\kappatup \in M$ ($\lambda(\kappatup)$ is the radius of
$p^{\kappa_m}\ball{1, \mu_m}$).
Note that $G_{\kappatup} = \{\xtup \in G \mid \vtup(\xtup) = \kappatup\}$ and
that $G$ is definable (using e.g.\ Lemma~2.1 of \cite{Den:cell}).

The claim we will prove by induction is the following. For $N$ sufficiently large,
there exists a definable set $X = X(\tree, \mutup) \subset \Zp^{m+N}$
of dimension at most $m + d$ such that the following holds:
\begin{itemize}
\item
$X \subset \bigcup_{\kappatup \in M} \left( G_{\kappatup} \times 
p^{\lambda(\kappatup)}\Zp^N\right)$
\item
For any $\kappatup \in M$ and any $\xtup \in G_{\kappatup}$,
$\Tr_{0, \lambda(\kappatup)}(X_{\xtup}) \cong \tree(\kappatup)$.
\item
For any $\kappatup \in M$,
the fiber $X_{\xtup}$ varies Lipschitz continuously with $\xtup \in
G_{\kappatup}$.
\end{itemize}

If $m = 0$, then $G = G_{\kappatup}$ is the one-point set,
where $\kappatup$ is the empty
tuple, $\lambda(\kappatup) = 0$, and the statement becomes $\Tr(X) \cong \tree$,
which is our theorem.

\medskip

Let $\joint_0, \dots, \joint_r$ be the joints of $\tree$,
including the virtual ones (i.e.\ the ones at depth infinity). We will
start by constructing definable functions $f_0, \dots, f_r\colon G \to \Zp^N$
which yield the skeleton of $\tree$ in the following sense.
For $\kappatup \in M$ and $\xtup \in G_{\kappatup}$,
set
\[
\tree_{\xtup} := \{\ball{f_i(\xtup), \lambda(\kappatup) + \nu}
\mid 0 \le i \le r, 0 \le \nu \le \depth(\joint_i)(\kappatup), \nu < \infty\}
\subset \Tr_{0, \lambda(\kappatup)}(\Zp^N)
.
\]
There will be isomorphisms
$\psi_{\xtup}\colon \tree(\kappatup) \to \tree_{\xtup}$
sending $\joint_i(\kappatup)$ to $\ball{f_i(\xtup), \lambda(\kappatup) +
\depth(\joint_i)(\kappatup)}$.

Let $X'$ be the union of the graphs of those functions
$f_i$ which correspond to virtual joints; the tree
$\Tr_{0,\lambda(\kappatup)}(X'_{\xtup})$ is exactly the subtree of $\tree_{\xtup}$
consisting of the infinite paths.
Later, we will define a set $X''$ which yields the side branches of $\tree$:
$X''$ will be a union
\[
X'' = \bigcup_{\kappatup \in M} \bigcup_{v \in \tree(\kappatup)}
X''_{\kappatup, v}
\]
such that for any $\xtup \in G_{\kappatup}$,
the fiber $Z := (X''_{\kappatup, v})_{\xtup}$ is contained in
the corresponding node $B := \psi_{\xtup}(v)$ of $\tree_{\xtup}$,
its tree $\Tr_B(Z)$ is isomorphic to the side branch of
$\tree(\kappatup)$ at $v$, and the intersection of $\Tr_B(Z)$ and
$\tree_{\xtup}$ consists only of $B$. We then set $X := X' \cup X''$.
Thus $\Tr_{0,\lambda(\kappatup)}(X_{\xtup})$ will have a side branch
at $B \in \tree_{\xtup}$ which is isomorphic to the corresponding one of
$\tree(\kappatup)$, and as $\tree(\kappatup)$ has no leaves,
$\Tr_{0,\lambda(\kappatup)}(X_{\xtup})$ will contain the whole skeleton
$\tree_{\xtup}$.

We will have to ensure that the fibers $X_{\xtup}$ vary
Lipschitz continuously with $\xtup \in G_{\kappatup}$.
Our functions
$f_i$ will satisfy 
\begin{equation}\label{eq:fLip}
v(f_i(\xtup_1) - f_i(\xtup_2)) \ge v(\xtup_1 - \xtup_2)
\quad \text{for } \xtup_1,\xtup_2 \in G_{\kappatup}
;
\end{equation}
this implies 
Lipschitz continuity of the fibers of $X'$. We will also prove
Lipschitz continuity for each set $X''_{\kappatup,v}$.
Then the Lipschitz union argument yields continuity for $X$.

Now let us construct the functions $f_i$.
To get the isomorphism $\tree(\kappatup) \cong \tree_{\xtup}$, it suffices to have
\begin{equation}\label{eq:fDif}
v(f_i(\xtup) - f_j(\xtup)) = d_{i,j}(\kappatup) + \lambda(\kappatup)
, 
\end{equation}
where $d_{i,j}\colon M \to \Gamma$
is the ``separating depth'' of the joints $\joint_i$ and $\joint_j$:
the depth of the deepest common ancestor of both.
Set $f_0(\xtup) := 0$ for all $\xtup \in G$.
For $j \ge 1$, consider the maximum $\dmax := \max\{d_{i,j}\mid 0 \le i < j\}$
under the partial order defined by pointwise comparison;
note that for $j$ fixed, all $d_{i,j}$ are comparable.
Choose any $i < j$ with $d_{i,j} = \dmax$ and define
\begin{equation}\label{eq:fDef}
f_j(\xtup) := f_i(\xtup) + \pow_{d_{i,j} + \lambda}(\xtup)\cdot (0,  
\underset{\text{pos. $i+1$}}{\underset{\uparrow}{\dots,0,1,0,\ldots}}, 0)
,
\end{equation}
where $\pow_{d_{i,j} + \lambda}$ comes from Lemma~\ref{lem:invFkt}.
By definition of $\pow_{d_{i,j} + \lambda}$, \req{fDef} implies \req{fDif}
for those specific $i,j$. For other pairs $i < j$, \req{fDif} follows by induction
on $j$. Moreover, \req{fLip} follows from the second property of the functions
$\pow_{d_{i,j} + \lambda}$.

\medskip

It remains to define the sets $X''_{\kappatup, v}$. We will show how to do
this when $v$ lies on a bone; for joints, a simplified version of the same
argument will do.
So fix a bone $\bone = (\joint_i, \joint_j)$ of $\tree$
and a congruence class $\ccGamma \in \Gamma/\rho\Gamma$.
Let $N_{\kappatup} := \{\kappa' \in \ccGamma \mid \depth(\joint_i)(\kappatup)< \kappa' < \depth(\joint_j)(\kappatup)\}$
be the set depths of the corresponding side branches of $\tree(\kappatup)$,
and set
$N := \{(\kappatup, \kappa') \mid \kappatup \in M, \kappa' \in N_{\kappatup}\}$.
We will construct a definable set
\[
Y = \bigcup_{\kappatup \in M} \!\bigcup_{%
\substack{v \in \bone(\kappatup)\\\depth(v) \in \ccGamma}}\!\!\!\!
X''_{\kappatup, v}
.
\]

For $\xtup \in G_{\kappatup}$, the fiber $(X''_{\kappatup, v})_{\xtup}$
is supposed to be contained in
$B := \psi_{\xtup}(v) = \ball{f_j(\xtup), \lambda(\kappatup) + \depth(v)}$.
By applying the isometry $(\xtup, \ytup) \mapsto (\xtup, \ytup - f_j(\xtup))$
(which neither harms the trees of fibers, nor Lipschitz continuity),
we may assume $f_j(\xtup) = 0$.

Now notice that in \req{fDef}, we did not use the first coordinate of $\Zp^N$
at all, hence any child of $B = p^{\lambda(\kappatup) + \depth(v)}\Zp^N$
in $\tree_{\xtup}$ is contained
in $p^{\lambda(\kappatup) + \depth(v)}(p\Zp \times \Zp^{N-1})$.
We will ensure that $\tree_{\xtup}$ and $\Tr_{B}((X''_{\kappatup, v})_{\xtup})$
only intersect in $B$ by choosing
\begin{equation}\label{eq:inA}
(X''_{\kappatup, v})_{\xtup} \subset
A_{\kappatup,v} := p^{\lambda(\kappatup) + \depth(v)}((1 + p\Zp) \times
\Zp^{N-1})
.
\end{equation}

Let $\fintree$ be the finite tree at the beginning of the side branch
of $\tree$ corresponding to $\bone, \ccGamma$,
and for each leaf $w$ of $\fintree$,
let $\tree_w\colon N \to \Trees$ be the corresponding side tree of level $d-1$.
Define a shifted set $\Nshift := \{(\kappatup, \lambda(\kappatup) + \kappa')
\mid (\kappatup, \kappa') \in N\}$ and a shifted tree
$\treeshift_w\colon \Nshift \to \Trees, \treeshift_w(\kappatup, \lambda(\kappatup) + \kappa') = \tree_w(\kappatup, \kappa')$.
We apply the induction hypothesis to $\treeshift_w$ using
$\mu_{m+1} := \depth_{\fintree}(w)$ (we may suppose $\depth_{\fintree}(w) > 0$);
denote by $X_w := X(\treeshift_w, (\mu_1, \dots, \mu_{m+1}))$ the resulting definable set.

Fix $\kappatup \in M$ and $\xtup \in G_{\kappatup}$.
For $z \in \Qp$, the fiber $(X_w)_{(\xtup,z)}$ is non-empty if and only
if $z \in p^{\lambda(\kappatup) + \kappa'}\ball{1,\mu_{m+1}}$ for some $\kappa' \in N_{\kappatup}$,
and if this is the case, then
$\Tr_{0, \lambda(\kappatup) + \kappa'+\mu_{m+1}}((X_w)_{(\xtup,z)}) \cong
\tree_w(\kappatup, \kappa')$.
Set 
\[
B_{\kappa'} :=
p^{\lambda(\kappatup) + \kappa'}\ball{(1,0,\dots,0), \depth_{\fintree}(w)}
\subset \Zp^N
;
\]
then $(X_w)_{\xtup}$ is contained in
$\bigcup_{\kappa' \in N_{\kappatup}} B_{\kappa'}$, and
Lipschitz continuity of fibers $(X_w)_{(\xtup,z)}$ of $(X_w)_{\xtup}$ yields
$\Tr_{B_{\kappa'}}((X_w)_{\xtup}) \cong \Tr(\Zp) \times \tree_w(\kappatup, \kappa')$.

Now choose an embedding of $\fintree$ into $\Tr(\Zp^{N-1})$
and let $\ball{\ystup_w, \depth(w)}$ be the image of the leaf $w$.
The map $\phi_w(\xtup, z, \ystup) := (\xtup, z, \ystup + z \cdot
\ystup_w)$ is an isometry sending $G_{\kappatup} \times B_{\kappa'}$ to
$G_{\kappatup} \times p^{\lambda(\kappatup) + \kappa'}\ball{(1, \ystup_w), \depth(w)}$.
We claim that the set $Y:= \bigcup_w \phi_w(X_w)$ is the one
we are looking for; more precisely, if $\kappatup \in M$, $v \in \bone(\kappatup)$,
$\kappa' := \depth(v) \in \ccGamma$, then we claim
\[
X''_{\kappatup,v} = \bigcup_w \phi_w(X_w \cap (G_{\kappatup} \times B_{\kappa'}))
.
\]
Fix $\xtup \in G_{\kappatup}$ and $B := p^{\lambda(\kappatup) + \kappa'}\Zp^N$.
$(X''_{\kappatup,v})_{\xtup}$ is contained in the union of balls
$B_{w} := p^{\lambda(\kappatup) + \kappa'}\ball{(1, \ystup_w), \depth(w)}$, which in
turn are contained in $A_{\kappatup,v}$, so \req{inA} is satisfied.

The finite subtree of $\Tr_B(\Zp^N)$ with leaves $B_w$ is isomorphic to $\fintree$,
and the tree of $(X''_{\kappatup,v})_{\xtup}$ on $B_w$ is isomorphic to
$\Tr(\Zp) \times \tree_w(\kappatup, \kappa')$, so the tree
$\Tr_{B}((X''_{\kappatup,v})_{\xtup})$ is the right one.
Finally, using Lipschitz continuity in $\xtup$ of the fibers of
$\phi_w(X_w \cap (G_{\kappatup} \times B_{\kappa'}))$
and the Lipschitz union argument, we get Lipschitz continuity of the fibers
of $X''_{\kappatup,v}$.

\medskip

Note that $\dim X' \le m$ and by induction $\dim X'' \le (m+1) + (d-1)$;
thus $\dim X \le m + d$.
\end{proof}

\section{The main proofs}
\label{sect:proofs}

In this section we will prove the main conjecture in the interesting cases.
We start by sketching the proofs; an overview over the remainder
of the section will be given after that sketch.

\subsection{Idea of proof}
\label{subsect:idea}

Suppose that $X$ is a definable set of dimension $d$ and that we want
to check that $\Tr(X)$ is a level $d$ tree. By compactness
(as in the case of smooth varieties) it suffices to understand
the tree on a neighborhood of each point of $\clX$.
To understand the tree near a given point---without loss $0$---we proceed
as in the example of the cusp curve: we compute it
on balls $B$
which are close to $0$ but which do not
contain $0$;
the largest such balls are of the form
$B = \ball{p^{\kappa}\xtup_0, \kappa + 1}$
with $v(\xtup_0) = 0$. The total tree will be of level $d$
if the following two conditions hold:
\begin{enumerate}
\item\label{it:cBall}
The tree on each ball $B$ looks like the tree of a side
branch: after cutting $B$ into finitely many smaller balls,
it is of the form $\Tr(\Zp) \times \tree$, where $\tree$
is of level $d - 1$.
\item\label{it:cUnif}
If we let $\kappa$ go to infinity (i.e.\ the ball $B$ approaches $0$),
then the trees on $B$ are uniform in $\kappa$
(in the way required by the definition of level $d$ trees).
\end{enumerate}

Now suppose that $X$ is one-dimensional.
For simplicity, assume moreover $X \subset \Qp^2$.
It is known that such a set $X$ is a subset of an algebraic set $V$.
By applying the theorem of Puiseux to $V$,
close to $(0,0)$
we can write $X$ as union of branches, each of which
is the graph of series of the form $f(x) = \sum_i a_i \sqrt[e]{x}^i$.
Taking the $e$-th root is of course not unique, but as in the cusp
example, on each ball $B = \ball{p^{\kappa}(x_0,y_0), \kappa + 1}$
we can choose roots in such a way that we get a continuous function $f$.
(In fact, here we might need to replace $\kappa+1$
by $\kappa+\mu$ for some fixed $\mu > 1$.)
Now suppose that $v(x_0) = 0$, i.e.\ $B$ does not lie 
directly above or below
$(0,0)$. Then for large $\kappa$, the graph of
$f$ will intersect $B$ only if its derivative at $0$ has non-negative
valuation. Using this, we get Lipschitz 
continuity of $f$: $v(f(x_1) - f(x_2)) \ge v(x_1 - x_2)$.
This will allow us to apply Corollary~\ref{cor:contTree}, which
will finally imply condition \rit{cBall}. If on the other hand $v(x_0) > 0$,
then $v(y_0) = 0$, and the same argument applies with coordinates exchanged.

All this can be carried out uniformly in
$\kappa$, and we will get the uniformity required in \rit{cUnif}
by having a second look at the Puiseux series describing the branches.
If $\sum_i a_i \sqrt[e]{x}^i$ is the difference of two such series,
then for $\kappa = v(x) \gg 0$, the valuation of this is equal to
$v(a_{\iota}) + \frac{\iota}{e}v(x)$, where $a_{\iota}$ is the first
non-zero coefficient. This valuation corresponds to the depth of
a joint of the side tree; as required, it is linear in $\kappa$.

\medskip

To get a proof for two-dimensional definable subsets of $\Qp^2$, we
use cell decomposition to understand $X$ and then
apply the Puiseux series arguments to the centers of cells (which are
curves). Lipschitz continuity of these centers yields
Lipschitz continuity of the whole fibers of the cells,
so Corollary~\ref{cor:contTree} implies that the tree
on a ball $B$ is of the form $\Tr(\Zp) \times \tree$, where
$\tree$ is the tree of one fiber.

Of course the tree of a fiber is of level $1$ (as its dimension
is at most $1$), but we need uniformity in $\kappa$. To prove this,
for each $\kappa$ we will choose one fiber $X_\kappa$ in the
corresponding ball. The cell decomposition of $X$ yields a cell decomposition
of each $X_\kappa$ which is ``close to uniform''; for example,
for $\kappa \gg 0$ a cell center will be close to $p^{\ell(\kappa)}\cdot a$
for some fixed $a \in \Qp$ and some linear function $\ell$.
This uniformity will allow
us to deduce that the parametrized tree $\kappa \mapsto \Tr(X_{\kappa})$
is of level $1$.

\medskip

The remainder of this section is organized as follows.
First, we recall cell decompositions; in the next two
subsections, we introduce ``garlands'', which are the right
sets to work on when one wants to carry out the above arguments
concerning Puiseux series
uniformly in $\kappa$. In Subsection~\ref{subsect:1dimParam},
we introduce the close-to-uniform families of sets $X_\kappa$
and prove that they have uniform level $1$ trees, and
in the last two subsections, we carry out the remainder of
the above arguments.

\subsection{Cell decomposition}


The following is almost the usual definition of a cell decomposition.
The only difference is that we are a bit more restrictive on the conditions
$\conda$ and $\condb$ in a harmless way;
this will save us a few clumsy case distinctions.

\begin{defn}\label{defn:cell}
\begin{enumerate}
\item
The only \emph{cell} in $\Qp^0$ is the one-point set $\Qp^0$ itself.

A \emph{cell} in $\Qp^n$ is a set of the form
\[
C =
\{(\xtup, y) \in D \times \Qp \mid
\alpha(\xtup) \conda v(y - c(\xtup)) \condb \beta(\xtup)
\text{ and }\exists z\; y - c(\xtup) = rz^e\}
,
\]
where $D$ is a cell in $\Qp^{n-1}$,
$\alpha, \beta\colon D \to \Gamma \cup \{\infty\}$ and $c\colon D\to \Qp$ are definable functions,
$r \in \Qp\mult$, $e \in \bbN_{\ge1}$,
$\conda$ is either $\le$ or no condition and $\condb$ is either $\le$ or $<$.
Moreover, we suppose that the projection $C \to D$ is surjective and
that if $\condb$ is $<$, then $\beta = \infty$.

We call $D$ the \emph{base}, $c$ the \emph{center},
$\alpha$ and $\beta$ the \emph{lower} and \emph{upper bound}, $e$ the \emph{exponent}
and $r$ the \emph{residue} of $C$.
\item
A \emph{cell decomposition of $\Qp^n$} is a partition of $\Qp^n$
into finitely many disjoint cells.
If
$n > 0$, then we additionally require that the set of bases of
the cells is a cell decomposition of $\Qp^{n-1}$.
\end{enumerate}
\end{defn}

By \emph{fixing a cell decomposition}, we will mean that we also fix
the data $D, c, \alpha, \beta, \dots$ describing the cells.

The usual cell decomposition theorem is the following; see e.g.\ \cite{SD:dim},
Section~4.

\begin{lem}\label{lem:cell}
Let $X \subset \Qp^n$ be a definable set. Then there
exists a cell decomposition of $\Qp^n$ such that $X$
is a union of cells.
\end{lem}

The following easy fact about one-dimensional cells will be used quite often:
\begin{lem}\label{lem:cellEasy}
There exists a function $\delta \colon \bbN_{\ge1} \to \Gamma_{>0}$
such that the following holds.
\begin{enumerate}
\item
Let $C \subset \Qp$ be a cell with center $c$ and exponent $e$,
and suppose $x_1 \in C$ and $x_2 \in \Qp \ohne C$.
Then $v(x_1 - x_2) < v(x_i - c) + \delta(e)$ for $i \in \{1,2\}$.
\item
Suppose that $C_1$ and $C_2$ are two disjoint cells
with centers $c_1$ and $c_2$ and common exponent $e$,
and suppose that $x_1 \in C_1$ and $x_2 \in C_2$.
Then $v(x_1 - x_2) < v(c_1 - c_2) + \delta(e)$.
\end{enumerate}
\end{lem}
\begin{proof}
Set $\delta(e) := 2v(e) + 1$. Then (1) follows from Lemma~\ref{lem:root} (2).

For (2), use
(1) and the disjointness of $C_1$ and $C_2$ to get (for $i=1,2$)
$v(x_1 - x_2) < v(x_1 - c_i) + \delta$.
Now apply the triangle inequality to $c_1, x_1, c_2$.
\end{proof}

\subsection{Garlands and trees}

Suppose that $X \subset \Zp^n$, $\xtup_0 \in \Zp^n$,
$B_0=\ball{\xtup_0, \lambda}$, and $B \subset B_0$ is a ball not containing
$\xtup_0$.
As described in Subsection~\ref{subsect:idea},
we will try to understand $\Tr_{B}(X)$ uniformly
when $B$ approaches $\xtup_0$. To be able to speak about uniformity,
we have to determine the trees on a whole ``garland'' of balls approaching
$\xtup_0$ at once. In this subsection, we define these garlands
and show that indeed knowing the trees on appropriate garlands suffices to
get back the whole tree of $X$ (Lemma~\ref{lem:compact}).

The reason to work on garlands and not on the whole of $B_0$ is
essentially that on a garland, it makes sense to speak of one
specific branch of the $e$-th root function, whereas on the whole of
$B_0$ it does not. In the next subsection, we will use this to
infer a nice description of definable functions
on garlands close to $\xtup_0$.

\begin{defn}\label{defn:garland}
Suppose we have $\xtup_0 \in \Zp^n$, $\lambda \in \Gamma_{\ge 0}$,
and $\mu, \rho \in \Gamma_{>0}$.
A \emph{garland} $G$ corresponding to $\xtup_0, \lambda, \mu, \rho$ is
a set of the form
\[
G = \xtup_0 +
\bigcup_{\substack{\kappa \ge \lambda\\\kappa \in \ccGamma}}
    p^\kappa \ball{\xtup_G, \mu}
\]
for some $\xtup_G \in \Zp^n$ satisfying $v(\xtup_G) = 0$ and some $\ccGamma
\in \Gamma/\rho\Gamma$. We will write
\[
M(G) := \{\kappa \in \ccGamma \mid \kappa \ge \lambda\}
\]
for the set over which the union goes, and
call the subsets $G_\kappa := \xtup_0 + p^\kappa \ball{\xtup_G, \mu}$
for $\kappa \in M$
the \emph{components} of $G$.
\end{defn}
\begin{rem}
$G_\kappa$ consists of exactly those $\xtup \in G$ which satisfy
$v(\xtup - \xtup_0) = \kappa$.
\end{rem}
\begin{rem}
For fixed $\xtup_{0}, \lambda, \mu, \rho$,
garlands form a finite partition of $\ball{\xtup_{0}, \lambda} \ohne
\{\xtup_0\}$.
\end{rem}

We will not always specify $\xtup_0, \lambda, \mu, \rho$;
sometimes we just write
``garland for $\lambda, \mu, \rho$'',
``garland converging to $x_0$'' or ``garland on $\ball{x_0,\lambda}$''.
Moreover, most of the time we will not care for the precise
values of $\lambda, \mu, \rho$; we will only require the garlands
to be ``sufficiently fine'', i.e.\ each garland is a subset of a garland
for certain given $\lambda_0, \mu_0, \rho_0$. This is equivalent to
$\lambda \ge \lambda_0$, $\mu \ge \mu_0$ and $\rho_0 \mid \rho$.
This is also what we will mean by ``$\lambda, \mu, \rho$ sufficiently large'':
for $\rho$ interpret ``large'' multiplicatively.

\begin{defn}\label{defn:treeGarl}
Let $X$ be a subset of $\Zp^n$ and let $G$ be a garland
whose components are $G_{\kappa}$,
for $\kappa \in M := M(G)$. The \emph{tree of $X$ on $G$} is the
parametrized tree
\[
\Tr_{G}(X)\colon M \to \Trees,
\kappa \mapsto \Tr_{G_\kappa}(X)
.
\]
\end{defn}

\begin{lem}\label{lem:compact}
Let $X$ be a subset of $\Zp^n$.
Suppose that for each $\xtup \in \Zp^n$,
there are $\lambda, \mu, \rho$ such that for each garland $G$
(corresponding to $\xtup, \lambda, \mu, \rho$), the parametrized tree
$\Tr_{G}(X)$ is of the form $\kappa \mapsto \Tr(\Zp) \times \tree_{G}(\kappa)$, where
$\tree_{G}$ is piecewise a parametrized tree of level $d$.
Then $\Tr(X)$ is a tree of level $d+1$.
\end{lem}

\begin{proof}
First, for each $\xtup \in \Zp^n \ohne \clX$ we enlarge the corresponding $\lambda$
such that $B(\xtup, \lambda) \cap X = \emptyset$.
As in the proof of Theorem~\ref{thm:mainSmooth}
(Subsection~\ref{subsect:smooth}), using compacity of 
$\Zp^n$ it suffices to prove that the tree on each ball $B(\xtup, \lambda)$
is of level $d+1$;
the whole tree will then consist of a finite tree, with finitely many
of the trees $\Tr_{\xtup, \lambda}(X)$ attached to it.

Now fix $\xtup \in \Zp^n$, and let $\lambda, \mu, \rho$ be as in the
prerequisites (possibly with $\lambda$ enlarged); we compute
the tree $\Tr_{\xtup,\lambda}(X)$.
To simplify notation, suppose $\xtup = 0$. If $0 \notin \clX$,
then $\ball{0,\lambda} \cap X = \emptyset$ and there is nothing to do,
thus suppose now $0 \in \clX$.
This implies $\ball{0, \kappa} \in \Tr_{0,\lambda}(X)$ for all
$\kappa \ge \lambda$. We take this as skeleton for $\Tr_{0,\lambda}(X)$,
with a joint at $\ball{0,\lambda}$ and then a single infinite bone.
It remains to determine the side branches.

Consider a garland $G$ for $\lambda, \mu, 1$ (converging to $0$).
It is the union of finitely
many garlands $G_i$ for $\lambda, \mu, \rho$, and
$\Tr_{G}(X)(\kappa) = \Tr_{G_i}(X)(\kappa)$ if $\kappa \in M(G_i)$.
Recall that $\Tr_{G_i}(X)(\kappa) \cong \Tr(\Zp) \times \tree_{G_i}(\kappa)$
and define $\tree_G(\kappa) := \tree_{G_i}(\kappa)$
if $\kappa \in M(G_i)$. We get that $\tree_G$
is piecewise of level $d$ and
$\Tr_{G}(X)(\kappa) \cong \Tr(\Zp) \times \tree_{G}(\kappa)$.
In other words, we may without loss suppose $\rho = 1$.

For each garland $G$, we have a finite partition
of $\{\kappa \in \Gamma \mid \kappa \ge \lambda\}$ such that
$\tree_{G}$ is of level $d$ on each set of the partition.
We choose a partition of $\{\kappa \in \Gamma \mid \kappa \ge \lambda\}$
such that for each part $M$, $\tree_{G}$ is of level $d$
on $M$ for all garlands $G$. Now we
claim that there is a single side
branch datum describing the side branch of $\Tr_{B_j}(X)$
leaving the skeleton at $\ball{0, \kappa}$ for all $\kappa \in M$.

Let $\fintree_\kappa$ be the subtree of $\Tr_{0,\kappa}(X)$
consisting of those $B = \ball{\xtup, \kappa + \nu}$ with
$0 \le \nu \le \mu$ and $0 \notin B$. Equivalently, $\fintree_\kappa$
is the finite subtree of $\Tr_{0,\kappa}(\Zp^n)$ whose leaves are exactly
the components $G_\kappa$ of those garlands $G$ satisfying $G_\kappa \cap X \ne
\emptyset$. For $G$ fixed,
this non-emptiness does not depend on $\kappa$ (as long as $\kappa \in M$),
so for two different $\kappa, \kappa' \in M$,
the map
\[
\{\xtup \mid v(\xtup) = \kappa\} \to \{\xtup \mid v(\xtup) = \kappa'\},
\xtup \mapsto p^{\kappa' - \kappa}\xtup
\]
induces (using Lemma~\ref{lem:treeIso}) an isomorphism from $\fintree_{\kappa}$ to
$\fintree_{\kappa'}$ sending $G_{\kappa}$ to $G_{\kappa'}$.

Now the side branch of
$\Tr_{B_j}(X)$ at $\ball{0, \kappa}$ consists of
$\fintree_\kappa$, with $\Tr_{G_\kappa}(X)$ attached
to the leaf $G_\kappa \in \fintree_\kappa$ (for $G_\kappa \cap X \ne
\emptyset$). As $\Tr_{G_\kappa}(X) \cong \Tr(\Zp) \times \tree_{G}(\kappa)$
with $\tree_{G}$ of level $d$, this proves the claim.
\end{proof}

\subsection{Definable functions on garlands}

The main result of this subsection (Proposition~\ref{prop:defPuiseux})
is that on sufficiently fine
one-dimensional garlands, a definable function is given by
a branch of a Puiseux series. We start by giving a meaning
to a specific branch of the $e$-th root function.

\begin{defn}
Suppose $G \subset \Qp$ is a garland for
$0, \lambda, \mu, \rho$, and suppose $e \in \bbN_{\ge 1}$.
We say that $G$ is \emph{fine enough for $e$-th roots} if $\mu \ge 2v(e) + 1$
and $e \mid \rho$.
Suppose that this is the case. Then a \emph{uniform choice of $e$-th roots
on $G$} is a choice of $\sqrt[e]{x} \in \aclQp$ for each $x \in G$
such that for any $x, x' \in G$ we have
$\frac{\sqrt[e]{x}}{\sqrt[e]{x'}} \in p^{\Gamma}\cdot (1 + p^{v(e)+1}\Zp)$.
\end{defn}

If $G$ is fine enough for $e$-th roots, then uniform choices of $e$-th
roots on $G$ exist. For any $x \in G$ choose any root $\sqrt[e]{x}$.
Then for any $x' \in G$ we have
$\frac{x'}{x} \in p^{e\cdot\nu}\cdot (1 + p^{2v(e)+1}\Zp)$ for some
$\nu \in \Gamma$; thus by Lemma~\ref{lem:root} (1), $\frac{x'}{x}$ has
a root $z \in p^{\nu}\cdot (1 + p^{v(e)+1}\Zp)$. Set $\sqrt[e]{x'} : = \sqrt[e]{x} \cdot
z$.

By ``choosing an $e$-th root on $G$'',
we will mean choosing $\sqrt[e]{x}$ uniformly as described above.
When we ask a garland to be fine enough
for $e$-th roots, we will often implicitly choose such a root.

If $G$ converges to $x_0 \ne 0$, by choosing an $e$-th root on $G$
we mean choosing $\sqrt[e]{x - x_0}$ for $x \in G$ in an analogous way.

These uniformly chosen roots are Lipschitz continuous in the following
sense:

\begin{lem}\label{lem:chooseRoot}
Suppose $e \in \bbN_{\ge1}$ and $G$ is a garland converging to $0$
which is fine enough for $e$-th roots. If
$x, x' \in G$ satisfy $x \etwa[\delta + v(e)] x'$
for some $\delta \ge 1$, then $\sqrt[e]{x} \etwa \sqrt[e]{x'}$,
and more generally $\sqrt[e]{x}^\iota \etwa \sqrt[e]{x'}^\iota$
for any $\iota \in \bbZ$.
\end{lem}
\begin{proof}
$x \etwa[\delta + v(e)] x'
\Leftrightarrow \frac{x}{x'} \in 1+p^{\delta+v(x)}\Zp
\Rightarrow \frac{\sqrt[e]{x}}{\sqrt[e]{x'}} \in 1+p^{\delta}\Zp
\Rightarrow \left(\frac{\sqrt[e]{x}}{\sqrt[e]{x'}}\right)^\iota \in
1+p^{\delta}\Zp
\Leftrightarrow\sqrt[e]{x}^\iota \etwa \sqrt[e]{x'}^\iota$.
\end{proof}

Note that if $x,x'$ lie in the same component of $G$
(and $G$ is fine enough for $e$-th roots),
we may always apply the lemma with
$\delta = v(x - x') - v(x) - v(e) \ge 1$.

\medskip

We will need the following two results relating garlands
and definable sets.

\begin{lem}\label{lem:garlandDef}
\begin{enumerate}
\item
Garlands are definable.
\item
If we chose an $e$-th root on a garland $G \subset \Zp$ and this
root lies in $\Qp$,
then $x \mapsto \sqrt[e]{x}$ is definable.
\end{enumerate}
\end{lem}

Note that whether $\sqrt[e]{x}$ lies in $\Qp$ does not depend
on the specific $x \in G$.

\begin{proof}[Proof of Lemma~\ref{lem:garlandDef}]
(1)
Well known; see e.g.\ \cite{Den:cell}, Lemma~2.1, 3) and 4).

(2)
We only need to specify in a definable way which of the roots we want to
take. If $z_0$ is the root of one element of $G$, then
the other ones are exactly the ones lying in
$z_0 \cdot p^{\Gamma}\cdot \ball{1,p^{v(e) + 1}}$. This is definable by
the same argument as for (1).
\end{proof}

\begin{lem}\label{lem:1dimLokal}
Let $X \subset \Qp$ be definable and $x_0 \in \Qp$.
Then there exist $\lambda, \mu, \rho$ such that any corresponding
garland converging to $x_0$ lies either completely inside or
completely outside of $X$.
\end{lem}

\begin{proof}
It is enough to prove the statement
when $X$ is a cell. If $x_0$ is not equal to the center of the cell,
or if the cell has an upper bound $\beta < \infty$, then
a whole ball $\ball{x_0, \lambda}$ lies either completely inside or completely
outside of $X$. Otherwise
choose $\lambda > \alpha$ (the lower bound) and use that the $e$-th
power residue on sufficiently fine garlands is constant.
\end{proof}

The two principal ingredients to our description of definable functions
on sufficiently fine garlands are
a lemma of Scowcroft and van den Dries which will allow us to
replace definable functions by branches of algebraic sets, and
the theorem of Puiseux which will allow us to describe such branches
in terms of branches of root functions.

\begin{lem}[Lemma 1.2 of \cite{SD:dim} and comment following its proof]
\label{lem:def2alg}
For any definable $X \subset \Qp$ and any definable function
$f \colon X \to \Qp$, the graph of $f$
is a subset of an algebraic curve.
\end{lem}

\begin{lem}[Theorem of Puiseux; see e.g.\ \cite{Eich:alg}, III.1.6]\label{lem:puiseux}
Let $V(\Qp) \subset \Qp^2$ be an algebraic curve.
Then there exists $\lambda \in \Gamma$,
a finite index set $N$, integers $e_{\nu} \ge 1$
and coefficients $a_{\nu,i} \in \aclQp$ for $i \in \bbZ$ and $\nu \in N$,
such that the following holds.
\begin{enumerate}
\item
For each $\nu \in N$, $a_{\nu,i} = 0$ for $i \ll 0$, and the Laurent series
\[
g_\nu(z) = \sum_{i \in \bbZ}a_{\nu,i}z^i
\]
converges for any $z \in \aclQp$ satisfying $v(z^{e_\nu}) \ge \lambda$.
\item
For any $(x,y) \in p^\lambda \Zp \times \Qp$, we have
$(x,y) \in V(\Qp)$
if and only if there exists a $\nu \in N$ and a root $\sqrt[e_\nu]{x} \in \aclQp$
such that $y = g_{\nu}(\sqrt[e_\nu]{x})$.
\end{enumerate}
\end{lem}

Now here is the main result of this subsection.

\begin{prop}\label{prop:defPuiseux}
Let $D \subset \Qp \ohne \{0\}$ be definable and let $f\colon D \to \Qp$ be a definable function.
Then there are $e, \lambda, \mu, \rho$
such that $D \cap \ball{0,\lambda}$ is a union of garlands corresponding to
$0, \lambda, \mu, \rho$, and
such that for each such garland $G \subset D$ the following holds.
$G$ is fine enough for $e$-th roots, and $f$ can be written as a 
convergent Laurent series in $\sqrt[e]{x}$, with coefficients $a_i \in \aclQp$:
\[
f(x) = \sum_{i \in \bbZ}a_i\sqrt[e]{x}^i
\]
for all $x \in G$.
\end{prop}

Note that the specific choice of an $e$-th root on $G$ does not matter;
to compensate for a change of root, multiply each $a_i$ by an appropriate
power of an $e$-th root of unity.

\begin{proof}
Choose $\lambda, \mu, \rho$ large enough such that $D \cap \ball{0,\lambda}$
is a union of corresponding garlands converging to $0$
(use Lemma~\ref{lem:1dimLokal}).
Let $V(\Qp) \subset \Qp^2$ be the algebraic curve
containing the graph of $f$ according to Lemma~\ref{lem:def2alg},
and apply Lemma~\ref{lem:puiseux} to $V$. Enlarge $\lambda$ such that
the conclusion of Lemma~\ref{lem:puiseux} holds on $\ball{0,\lambda}$.
Then for any $x \in D \cap \ball{0,\lambda}$, there exists a $\nu \in N$ and
an $e_\nu$-th root of $x$ such that
\[
f(x) = \sum_{i \in \bbZ}a_{\nu,i}\sqrt[e_\nu]{x}^i
.
\]
This statement remains true if we replace all $e_\nu$ by their least common
multiple and renumber the coefficients $a_{\nu,i}$ accordingly.

Now choose a primitive $e$-th root of unity $\zeta$,
enlarge $\mu$ and $\rho$ such that corresponding garlands
are fine enough for $e$-th roots, and choose an $e$-th root on each of them.
Define the set of formal Laurent series
\[
S := \left\{ \sum_{i \in \bbZ}a_{\nu,i}(\zeta^j \sqrt[e]{x})^i
  \in \aclQp[[\sqrt[e]{x}]] \,\bigg|\,  \nu \in N, 0 \le j < e
     \right\}
,
\]
and for $G \subset D$ and $s \in S$, set
$A_{G,s} := \{x \in G \mid f(x) = s(x)\}$. The union of these sets
is equal to $D \cap \ball{0,\lambda}$. We claim that after enlarging $\lambda$,
we may suppose that they are definable and disjoint.

For $s = \sum_{i}b_i\sqrt[e]{x}^i \in S$,
let $s_\tau := \sum_{i \le \iota}b_i\sqrt[e]{x}^i$ be the
corresponding truncated series, where $\iota$ is large enough such that
$s \ne s'$ implies $s_\tau \ne s'_\tau$ for any $s, s' \in S$.
Then for $v(x) \gg 0$, we have $v(s(x) - s_\tau(x)) > v(s_\tau(x) - s'_\tau(x))$
for any two different $s, s' \in S$, so we get that
$x \in A_{G,s}$ if and only if
$x \in G$ and $v(f(x)-s_\tau(x)) > v(f(x)-s'_\tau(x))$ for all $s' \in S \ohne
\{s\}$. This condition is definable and implies disjointness.

So now we have a finite definable partition $(A_{G,s})$ of $D \cap \ball{0,\lambda}$.
To finish the proof, enlarge $\lambda, \mu, \rho$ again such that 
any of the finer garlands is completely contained in one of the
sets $A_{G,s}$; on each of those finer garlands we have
$f(x) = s(x) = \sum_i b_i \sqrt[e]{x}$.
\end{proof}

We will need an analogue of the previous proposition for 
definable functions going to $\Gamma \cup \{\infty\}$;
we get it as a corollary of the previous proposition, although
the heavy machinery of Proposition~\ref{prop:defPuiseux} is not really
necessary. (It could, for example, also be deduced from
Corollary~6.5 of \cite{Den:rat} together with Lemma~\ref{lem:1dimLokal}.)

\begin{cor}\label{cor:defPuiseux}
Let $D \subset \Qp$ be a definable set and $\alpha\colon D \to \Gamma \cup \{\infty\}$
a definable function. Then there are $\lambda, \mu, \rho$
such that
on each garland $G \subset D$ corresponding to
$0, \lambda, \mu, \rho$, $\alpha(x)$ only depends on $v(x)$, and the function
$M(G) \to \Gamma \cup \{\infty\}, v(x) \mapsto \alpha(x)$ is linear.
\end{cor}

\begin{proof}
Write $\alpha$ as $v \circ f$ for some definable $f\colon D \to \Qp$.
Apply Proposition~\ref{prop:defPuiseux} to get $f(x) = \sum_{i}a_i\sqrt[e]{x}^i$, and
let $\iota$ be minimal such that $a_\iota \ne 0$.
If $v(x)$ is sufficiently large, then
$v(f(x)) = v(a_\iota\sqrt[e]{x}^\iota) = v(a_\iota) + \frac{\iota}{e}v(x)$,
so choose $\lambda$ accordingly.
\end{proof}

\medskip

To conclude this subsection, we prove to two general statements on Puiseux series
which we will need later.

\begin{lem}\label{lem:series}
Suppose that $G$ is a garland for $0, \lambda, \mu, \rho$ which is fine enough
for $e$-th roots and that the Laurent series
\[
f(x) = \sum_{i \in \bbZ} a_i\sqrt[e]{x}^i
\]
(with coefficients $a_i \in \aclQp$) converges on $G$.
\begin{enumerate}
\item\label{it:sQp}
If $f(x) \in \Qp$ for all $x \in G$, then
$a_i \sqrt[e]{x} \in \Qp$ for all $x \in G$ and all $i \in \bbZ$.
\item\label{it:sflat}
If $v(f(x)) \ge v(x)$ for all
$x \in G$, then there exists a $\lambda' \ge \lambda$
such that for all $x_1, x_2 \in G$ with $v(x_1) = v(x_2) \ge \lambda'$,
we have $v(f(x_2) - f(x_1)) \ge v(x_2 - x_1)$.
\end{enumerate}
\end{lem}

\begin{proof}
(1)
As $\frac{\sqrt[e]{x'}}{\sqrt[e]{x}} \in \Qp$ for any $x,x' \in G$,
it suffices to check the claim for one single $x \in G$.
Now suppose that $\iota$ is minimal
such that $a_\iota\sqrt[e]{x}^\iota \notin \Qp$.
For $y \in \aclQp$, write $\dQp(y) := \sup\{v(y-y')\mid y'\in \Qp\}$
for the distance of $y$ to $\Qp$.
As $\Qp$ is closed in $\aclQp$ in the $p$-adic topology, we have
$\dQp(a_\iota\sqrt[e]{x}^\iota) > 0$.

As $\frac{a_{\iota}\sqrt[e]{x'}^\iota}{a_{\iota}\sqrt[e]{x}^\iota} \in \Qp$
for any  $x, x' \in G$, we have
$\dQp(a_\iota\sqrt[e]{x}^\iota) = v(a_\iota\sqrt[e]{x}^\iota) + d_0$
for some fixed $d_0 \in \Gamma$ not depending on $x \in G$.
Thus, for $x$ sufficiently close to zero, we get
$v(a_i\sqrt[e]{x}^i) > \dQp(a_\iota\sqrt[e]{x}^\iota)$ for all $i > \iota$.
Together with $\sum_{i<\iota}a_i\sqrt[e]{x}^i \in \Qp$, this contradicts
$\sum_{i\in \bbZ}a_i\sqrt[e]{x}^i \in \Qp$.

\medskip

(2)
Suppose that $a_\iota$ is the first non-zero coefficient of the series.
The condition $v(f(x)) \ge v(x)$ (applied to sufficiently small $x$)
implies that $\iota \ge e$, and if $\iota = e$, then
$v(a_\iota) \ge 0$.

Now suppose $x_1, x_2 \in G$ are given. The claim $v(f(x_2) - f(x_1)) \ge v(x_2 - x_1)$
follows if we can verify the inequality
\begin{equation}\label{eq:summandFlat}
v(a_i\sqrt[e]{x_2}^i - a_i\sqrt[e]{x_1}^i) = 
v(a_i) + v(\sqrt[e]{x_2}^i - \sqrt[e]{x_1}^i) \ge v(x_2 - x_1)
\end{equation}
for all $i \ge \iota$.

If $i = e = \iota$, then $\sqrt[e]{x_2}^i - \sqrt[e]{x_1}^i = x_1 - x_2$, so \req{summandFlat}
follows from $v(a_i) \ge 0$. Now suppose $i > e$.

Set $\sigma := v(x_2 - x_1) - v(x_1)$. By Lemma~\ref{lem:chooseRoot},
we get $\sqrt[e]{x_1}^i\etwa[\sigma-v(e)]\sqrt[e]{x_2}^i$.
So
\[
v(\sqrt[e]{x_2}^i - \sqrt[e]{x_1}^i)
\ge v(\sqrt[e]{x_1}^i) + \sigma - v(e) = \tfrac{i}{e}v({x_1}) + v(x_2 - x_1) - v(x_1) - v(e)
,
\]
and it remains to verify $v(a_i) + \frac{i}{e}v({x_1}) - v(x_1) - v(e) \ge 0$.
This is true for $v(x_1) \gg 0$, but we need a bound which is independent
of $i$.

Choose any $x_0 \in G$ and set $\lambda_0 := v(x_0)$.
Let $i_0 \in \bbZ$ be such that
$v(a_{i_0}) + \frac{{i_0}}{e}\lambda_0$ is minimal (a minimum exits by convergence
of $f(x_0)$).
By supposing $v(x_1) \ge \lambda_0$, we get
\[
\begin{aligned}
&v(a_i) + \tfrac{i}{e}v({x_1}) - v(x_1) - v(e)\\
=\,\,& v(a_i) + \tfrac{i}{e}\lambda_0 \,\,+\,\, \tfrac{i}{e}\left(v({x_1}) - \lambda_0\right) \,\,- v(x_1) - v(e)
\\
\ge\,\,& v(a_{i_0}) + \tfrac{{i_0}}{e}\lambda_0 \,\,+\,\, \tfrac{e+1}{e}\left(v({x_1}) - \lambda_0\right) \,\,- v(x_1) - v(e)
\\
=\,\,& v(a_{i_0}) + \tfrac{{i_0}}{e}\lambda_0 - \tfrac{e+1}{e}\lambda_0 - v(e) + \tfrac{1}{e}v(x_1).
\\
\end{aligned}
\]
Now everything is constant except for the last summand,
so for $v(x_1)$ sufficiently large, this is non-negative.
\end{proof}

\subsection{Parametrized subsets of \texorpdfstring{$\Qp$}{\041\032p}}
\label{subsect:1dimParam}

For (one-dimensional) subsets of $\Qp$, the main conjecture is not
difficult to prove:

\begin{lem}\label{lem:1dim}
If $X$ is a definable subset of $\Qp$,
then $\Tr(X)$ is of level $1$.
\end{lem}
\begin{proof}
By Lemma~\ref{lem:1dimLokal}, trees on sufficiently fine garlands
close to any given point are isomorphic either
to $\kappa \mapsto \emptyset$ or to $\kappa \mapsto \Tr(\Zp)$,
so in any case they are of the form $\kappa \mapsto \Tr(\Zp) \times \tree(\kappa)$
where $\tree$ is of level $0$.
Thus Lemma~\ref{lem:compact} yields that the total tree $\Tr(X)$
is of level $1$.
\end{proof}

To prove the conjecture for definable subsets of $\Qp^2$, we will need a
parametrized version of this: if we have
definable sets $X_{\kappa} \subset \Qp$ parametrized by $\kappa \in \Gamma$
in a suitable ``uniform'' way, then we should get a parametrized level
$1$ tree. To state this, we need a notion of
``sufficient uniform maps'' from $\Gamma$ to $\Qp$.

\begin{defn}\label{defn:deltaUniform}
Let $\delta \in \Gamma_{>0}$, $M \subset \Gamma_{\ge0}$
and $c_\kappa \in \Qp$ for $\kappa \in M$.
We say that $\kappa \mapsto c_\kappa$ is \emph{$\delta$-uniform},
if $\kappa \mapsto v(c_\kappa)$ is
linear and if there exists an $a \in \Zp\mult$ such that
$c_\kappa \etwa p^{v(c_\kappa)}a$ for all $\kappa \in M$.
\end{defn}

Now here is a uniform version of Lemma~\ref{lem:1dim}.

\begin{prop}\label{prop:1dimParam}
Suppose that for each $\kappa$ in a subset $M \subset \Gamma_{\ge0}$ we
are given a definable set $X_\kappa \subset \Qp$, and that these sets are
uniform in $\kappa$ in the following sense. Each $X_\kappa$ is the union
of finitely many disjoint cells $C_{\kappa,i}$, $i \in I$ of the form
\[
C_{\kappa,i} = \{x \in \Qp \mid \alpha_{\kappa,i} \conda_{i} v(x - c_{\kappa,i}) \condb_{i} \beta_{\kappa,i}
\text{ and }\exists z\; x - c_{\kappa,i} = r_{i}z^{e}\}
.
\]
We require that all exponents are equal and
that none of the index set $I$, the exponent $e$, the residues $r_{i}$
and conditions $\conda_i$, $\condb_i$ depend on $\kappa$.
Moreover set $\delta := \delta(e)$ as in Lemma~\ref{lem:cellEasy}.
We require that for each $i,j \in I$, the functions $\kappa \mapsto \alpha_{\kappa,i}$ and
$\kappa \mapsto \beta_{\kappa,i}$ are linear,
and the functions $\kappa \mapsto c_{\kappa,i}$ and $\kappa \mapsto
c_{\kappa,i}-c_{\kappa,j}$ are $\delta$-uniform.

Under these conditions on $X_\kappa$, the tree
$M \to \Trees, \kappa \mapsto \Tr(X_\kappa)$ is piecewise
a parametrized level $1$ tree.
\end{prop}

Note that the requirement that the exponents of all cells are equal
is not a real restriction: anyway cell decompositions can be refined such that
all exponents become equal.

Before we start with the proof, let us state a variant
as a corollary.

\begin{cor}\label{cor:1dimParam}
Suppose that $M \subset \Gamma_{\ge0}$ and $X_\kappa$ (for $\kappa \in M$)
are given as in Proposition~\ref{prop:1dimParam} and satisfy all the conditions
required there with exception of the uniformity condition on the
cell centers $c_{\kappa,i}$. (We do however still require the uniformity of
differences $c_{\kappa,i} - c_{\kappa,j}$.)
Suppose moreover that
$B_\kappa = \ball{b_\kappa, \sigma_\kappa}$ are balls, where
the function of radii $\kappa \mapsto \sigma_\kappa$ is linear
and such that for any $i \in I$,
the function $\kappa \mapsto c_{\kappa,i} - b_\kappa$ is $\delta$-uniform
(with $\delta$ as in the proposition). Then the tree
$M \to \Trees, \kappa \mapsto \Tr_{B_\kappa}(X_\kappa)$ is piecewise
a parametrized level $1$ tree.
\end{cor}
\begin{proof}[Proof of the corollary]
Define $\psi_\kappa(x) := p^{-\sigma_\kappa}(x - b_k)$.
Then $\Tr_{B_\kappa}(X_\kappa) \cong \Tr(\psi_\kappa(X_\kappa))$, so
it suffices to verify uniformity of the sets $\psi_\kappa(X_\kappa)$.
Uniformity of the cell bounds and
$\delta$-uniformity of differences of centers carries over
(by linearity of $\kappa \mapsto \sigma_\kappa$), and
$\delta$-uniformity of $\kappa \mapsto c_{\kappa,i} - b_\kappa$
yields $\delta$-uniformity of $\kappa \mapsto \psi_\kappa(c_{\kappa,i})$.
The exponent $e$ and the conditions $\conda_i$, $\condb_i$ do not change,
so it remains to consider the residues $r_i$. They are replaced by
$p^{-\sigma_\kappa}r_i$, which does depend on $\kappa$.
However, as we only want to prove piecewise uniformity of the resulting trees,
we may partition $M$ according to $\sigma_\kappa$ modulo $e$; on these
parts, the $e$-th power residue of $p^{-\sigma_\kappa}r_i$ is constant,
so we may replace $p^{-\sigma_\kappa}r_i$ by one fixed value.
\end{proof}

\begin{proof}[Proof of Proposition~\ref{prop:1dimParam}]
We may suppose that $M$ is infinite; otherwise the statement follows from
Lemma~\ref{lem:1dim}.

We will prove the statement inductively, starting from the leaves. We will cut the tree
horizontally into slices. There will be some thin ones where ``the things happen''
and some thick and simple parts in between where the skeleton of the tree
will only consist of long bones.
Let us make this precise.

By ``the involved linear functions'' we mean the set of maps from
$M$ to $\Gamma \cup \{\infty\}$ consisting of
$\kappa \mapsto \alpha_{\kappa,i}$, $\kappa \mapsto \beta_{\kappa,i}$,
$\kappa \mapsto v(c_{\kappa,i})$ and $\kappa \mapsto v(c_{\kappa,i} -
c_{\kappa,j})$ for $i, j\in I$.

For two linear functions $\ifu_1, \ifu_2 \colon M \to \Gamma \cup \{\infty\}$, we write
\[
\ifu_1 \ll \ifu_2 \defiff \lim_{\kappa \to \infty}
\ifu_2(\kappa) - \ifu_1(\kappa) = \infty
.
\]
(If $\ifu_1$ and $\ifu_2$ both are constant $\infty$, we set $\ifu_1 \nll \ifu_2$.)
By treating finitely many elements of $M$ separately
using Lemma~\ref{lem:1dim}, we may suppose that if 
$\ifu_1$ and $\ifu_2$ both are either involved or constant $0$, then
\begin{equation}\label{eq:ifuSort}
\ifu_1 \ll \ifu_2 \Longrightarrow
\ifu_2(\kappa) - \ifu_1(\kappa) \ge \max\{2\delta,e + 1\}
\text{ for all }\kappa \in M
.
\end{equation}
In particular, $\le$ defines a total order on the involved functions and the
zero function, and whether a cell center 
$c_{\kappa,i}$ lies in $\Zp$ is independent of $\kappa$.

By partitioning $M$ into finitely many definable sets and treating each one
separately,
we may suppose that moreover for any $i \in I$, whether or not $C_{\kappa,i} \cap \Zp$
is empty is independent of $\kappa$.
By removing cells not intersecting $\Zp$, we may suppose
$C_{\kappa,i} \cap \Zp \ne \emptyset$
for any $i \in I$ and any $\kappa \in M$.

\medskip

Our induction will run over the number of involved functions $\ifu$ satisfying
$\ifu \gg 0$. Thus by induction hypothesis,
we can apply Corollary~\ref{cor:1dimParam} to $(X_\kappa)_\kappa$ and a
family of
balls $B_\kappa = \ball{b_\kappa,\sigma_\kappa}$, provided there is at least
one involved function $\ell \gg 0$ such that $\ell - (\kappa\mapsto\sigma_\kappa) \ngg 0$.

We will now first treat the special case where
every lower bound $\alpha_{\kappa,i}$ satisfies either $\alpha_{\kappa,i} \le 0$ or
$\alpha_{\kappa,i} \gg 0$, and every other involved function
$\ifu$ satisfies $\ifu \gg 0$.
This corresponds to the thick but simple slices in our tree.
Afterwards we will reduce the general case to the first one; this reduction corresponds
to the thin but complicated slices.

\medskip

\textbf{The thick and simple parts:}\nopagebreak

Let $\ifu'_0$ be the minimal (with respect to $\le$) involved function satisfying $\ifu'_0\gg 0$,
and define $\ifu_0 := \ifu'_0 - \max\{\delta, e\}$.
By \req{ifuSort}, we have $\ifu_0(\kappa) > 0$ for all $\kappa \in M$.

We may suppose $I \ne \emptyset$. Choose an arbitrary $i_0 \in I$
and suppose without loss $c_{\kappa,i_0} = 0$ for all $\kappa \in M$.
Thus $v(c_{\kappa,i}) \ge \ifu_0(\kappa)+\delta$ for all $i \in I$.
Moreover, as $C_{\kappa,i} \cap \Zp$ is non-empty and
$\beta_{\kappa,i} \ge \ifu_0(\kappa)+e$, we get
$C_{\kappa,i} \cap \ball{0,\ifu_0(\kappa)} \ne \emptyset$;
so $\ball{0,\lambda} \in \Tr(X_\kappa)$ for all $\lambda \le \ifu_0(\kappa)$.

Now suppose first that $\ifu_0 < \infty$, and set $B_\kappa := \ball{0,
\ifu_0(\kappa)}$. The parametrized tree $\kappa \mapsto \Tr_{B_\kappa}(X_\kappa)$
is of level $1$ by induction hypothesis, as the involved function $\ifu'_0$
satisfies $\ifu'_0 \gg 0$ and $\ifu'_0 - \ifu_0 \ngg 0$.
By Lemma~\ref{lem:glueCheese}, it is therefore enough to verify
that the tree on the cheese $S_\kappa := \Zp \ohne B_\kappa$ is of level $1$ in such a way
that $\kappa \mapsto B_\kappa$ is a joint.
We choose $\{\ball{0,\lambda} \mid 0 \le \lambda \le \ifu_0(\kappa)\}$
as skeleton (with a single bone of length $\ifu_0$); it remains to
analyse the side branches.

If $\ifu_0 = \infty$, then we do not need the induction hypothesis; we
simply define $S_{\kappa} := \Zp$ and choose
$\{\ball{0,\lambda} \mid \lambda \ge 0\}$ as skeleton for $\Tr_{S_\kappa}(X_\kappa)$
(again with one single bone).

The tree $\Tr_{S_\kappa}(X_\kappa)$ does not change if we replace all centers
of cells $c_{\kappa,i}$ by $0$: if $\ifu_0 = \infty$, there is nothing to do;
otherwise this follows from Lemma~\ref{lem:cellEasy} (1),
using that for $x \notin B_\kappa$, we have $v(x - c_{\kappa,i}) < \ifu_0(\kappa)
\le v(c_{\kappa,i} - 0) - \delta$.
So for $x \in S_\kappa \ohne \{0\}$, we get that $x \in X_\kappa$ if and only if
there is an $i \in I$ with $\alpha_{\kappa,i} \le 0$
such that $\frac{x}{r_i}$ is an $e$-th power.
Thus for $\lambda < \ifu_0(\kappa)$,
the side branch of $\Tr(X_\kappa)$ at $\ball{0,\lambda}$ only depends on
$\lambda$ modulo $e$ and not on $\kappa$ at all. Moreover, each side
branch consists of a finite tree with copies of $\Tr(\Zp)$ attached to
its leaves; hence $\kappa \mapsto \Tr_{S_\kappa}(X_\kappa)$
is indeed of level $1$.

\medskip

\textbf{The thin and complicated slices:}\nopagebreak

(Reduction of the general case to the case
where all involved $\ifu$ satisfy $\ifu \gg 0$, except for lower bounds $\alpha_{\kappa,i}$
which may also be $\alpha_{\kappa,i} \le 0$.)

Let us first have a look at cells whose centers $c_{\kappa,i}$ lie
outside of $\Zp$. If $v(c_{\kappa,i}) < -\delta$ and
$C_{\kappa,i} \cap \Zp \ne \emptyset$, then
Lemma~\ref{lem:cellEasy} yields $\Zp \subset C_{\kappa,i}$, so this case is trivial. If
$-\delta \le v(c_{\kappa,i}) < 0$, then $v(c_{\kappa,i})$ does not depend
on $\kappa$ by \req{ifuSort}, and $\delta$-uniformity of $c_{\kappa,i}$ yields
$c_{\kappa,i} \etwa a'$ for some $a' \in \Qp$ not depending
on $\kappa$. Thus for any two different $\kappa, \kappa' \in M$,
we get $v(c_{\kappa,i} - c_{\kappa',i}) \ge 0$. Moreover,
$C_{\kappa,i} \cap \Zp \ne \emptyset$ implies
$\alpha_{\kappa,i} \le v(c_{\kappa,i}) = v(x-c_{\kappa,i}) \le \beta_{\kappa,i}$ for all
$\kappa$ and all $x \in \Zp$.
This yields bijections
\begin{equation}\label{eq:cAussen}
\Zp \cap C_{\kappa,i} \to \Zp \cap C_{\kappa',i}, x \mapsto x - c_{\kappa,i} + c_{\kappa',i}
\end{equation}
for all $\kappa, \kappa' \in M$, which will be useful later.

Now let $\lambda$ be the maximum
value of all constant involved functions.
We will cut out holes of radius $\lambda + \delta$ around the centers of
some of the cells, apply the  thick and simple case
to get the trees in these holes, compute the tree outside of the holes
and then put everything together. Define
$B_{\kappa,i} := \ball{c_{\kappa,i},\lambda + \delta}$
for $i \in I$. We do not want to cut out all $B_{\kappa,i}$,
but only those in which $X_\kappa$ is complicated:
define $J \subset I$ in such a way that $j \in J$ implies
$c_{\kappa,j} \in \Zp$ and
$C_{\kappa,j} \cap B_{\kappa,j} \ne \emptyset$. Moreover, if there are
several $i$ for which the balls $B_{\kappa,i}$ are equal, then put
only one representative into $J$.

Let us first analyse the relative position of a cell $C_{\kappa,i}$
and a hole $B_{\kappa,j}$ ($i \in I, j \in J$). We claim that
either $c_{\kappa,i} \in B_{\kappa,j}$ or $C_{\kappa,i} \cap B_{\kappa,j} =
\emptyset$, and that this does not depend on $\kappa$. Indeed, if
$v(c_{\kappa,i} - c_{\kappa,j}) \gg 0$, then by \req{ifuSort} we have
$v(c_{\kappa,i} - c_{\kappa,j}) \ge \lambda + 2\delta$ for all $\kappa\in M$,
so
$c_{\kappa,i} \in B_{\kappa,j}$. If on the other hand $v(c_{\kappa,i} - c_{\kappa,j}) \ngg 0$, then
$v(c_{\kappa,i} - c_{\kappa,j}) \le \lambda$ for all $\kappa\in M$,
and Lemma~\ref{lem:cellEasy} (1) implies that
$B_{\kappa,j}$ lies either completely inside or completely
outside of $C_{\kappa,i}$. As
$B_{\kappa,j} \cap C_{\kappa,j} \ne \emptyset$, the disjointness of
$C_{\kappa,i}$ and $C_{\kappa,j}$ implies $C_{\kappa,i} \cap
B_{\kappa,j} = \emptyset$.

Now fix $j \in J$. Computing the tree $\kappa
\mapsto \Tr_{B_{\kappa,j}}(X_\kappa)$ in the hole $B_{\kappa,j}$
can be done using the corollary version of the thick-and-simple case,
after removing all cells not intersecting $B_{\kappa,j}$.
Indeed, the required uniformity in $\kappa$ is clear, and the condition
$\ifu \gg \lambda + \delta$ for involved $\ell$
(or $\alpha_{\kappa,i} \le \lambda + \delta$ for lower bounds) follows
from the fact that $C_{\kappa,i} \cap B_{\kappa,j} \ne \emptyset$
implies $v(c_{\kappa,i} - c_{\kappa,j}) \gg 0$ and $\beta_{\kappa,i} \gg 0$.

By Lemma~\ref{lem:glueCheese} we are left to compute the
tree on the cheese $S_\kappa := \Zp \ohne \bigcup_{j \in J} B_{\kappa,j}$.
We will first check that for each $\kappa$ separately, the tree
$\Tr_{S_\kappa}(X_\kappa)$
is of level $1$ (with the nodes $B_{\kappa,j}$ being
joints), and then we will find isomorphisms
$\Tr_{S_\kappa}(X_\kappa) \cong \Tr_{S_\kappa}(X_{\kappa'})$ respecting the holes.
This implies that $\kappa \mapsto \Tr_{S_\kappa}(X_\kappa)$ is parametrized
of level $1$.

To prove that $\Tr_{S_\kappa}(X_\kappa)$ is of level $1$,
it is enough to show that any ball $B \subset S_\kappa$ of radius
$\lambda + 2\delta$ lies either completely inside or
completely outside of $X_\kappa$.
So suppose $x \in X_\kappa \cap S_\kappa$. Then $x \in C_{\kappa,i}$ for some $i \in I$,
and our choice of holes ensures that $v(x-c_{\kappa,i}) < \lambda + \delta$.
Lemma~\ref{lem:cellEasy} (1) implies that $C_{\kappa,i}$ (and therefore $X_\kappa$)
contains $\ball{x, \lambda + 2\delta}$.

To get the isomorphisms $\Tr_{S_\kappa}(X_\kappa) \to \Tr_{S_{\kappa'}}(X_{\kappa'})$ we
first replace (for each $\kappa$) $X_\kappa$ by a set $Y_\kappa$
which has the same tree on $S_\kappa$, but which is simpler inside the holes.
We ensure that $\Tr(Y_\kappa)$ contains the nodes $B_{\kappa,j}$, $j \in J$,
so that $\Tr_{S_\kappa}(Y_\kappa) \subset \Tr(Y_\kappa)$.
Then we will use
Lemma~\ref{lem:treeIso} to construct an isomorphism
$\Tr(Y_\kappa) \to \Tr(Y_{\kappa'})$ sending $B_{\kappa,j}$ to $B_{\kappa',j}$;
this yields the desired isomorphism
$\Tr_{S_\kappa}(X_\kappa) = \Tr_{S_\kappa}(Y_\kappa) \iso \Tr_{S_{\kappa'}}(Y_{\kappa'})
= \Tr_{S_\kappa}(X_{\kappa'})$.

Define
$Y_{\kappa} := (X_{\kappa} \cap S_{\kappa}) \cup \{c_{\kappa,j} \mid j \in J\}$.
It is clear that $\Tr_{S_\kappa}(X_\kappa) \cong \Tr_{S_{\kappa}}(Y_{\kappa})$,
and the element $c_{\kappa,j}$ ensures that $B_{\kappa,j}$ is a node of
$\Tr(Y_\kappa)$.
It remains to define the bijective isometry $\phi\colon Y_{\kappa} \to Y_{\kappa'}$
needed in Lemma~\ref{lem:treeIso}. To this end, let us first adapt our
cell decomposition to the sets $Y_\kappa$: define
\[
D_{\kappa,i} := C_{\kappa,i} \ohne \bigcup_{j \in J} B_{\kappa,j}
.
\]
Thus $X_\kappa \cap S_\kappa = \Zp \cap \bigcup_{i\in I}D_{\kappa,i}$.
Our choice of $J$ ensures that $D_{\kappa,i} = C_{\kappa,i} \ohne B_{\kappa,i}$
if $c_{\kappa,i} \in \Zp$ and $D_{\kappa,i} = C_{\kappa,i}$ otherwise,
so $D_{\kappa,i}$ is a cell again, 
and moreover $x \in \Zp \cap D_{\kappa,i}$ implies
$v(x - c_{\kappa,i}) < \lambda + \delta$.


Next, we claim that the map $x \mapsto x - c_{\kappa,i} + c_{\kappa',i}$
induces a bijection from $D_{\kappa,i} \cap \Zp$ to $D_{\kappa',i} \cap \Zp$.
If $c_{\kappa,i} \notin \Zp$, then this has already been verified in \req{cAussen}.
Otherwise, it follows from the fact that the bounds of $D_{\kappa,i}$ are either
independent of $\kappa$ or less than $0$.
Using this, we define the bijection $\phi\colon Y_{\kappa} \to Y_{\kappa'}$ by
$\phi(x) := x - c_{\kappa,i} + c_{\kappa',i}$ if $x \in D_{\kappa,i} \cap \Zp$,
$i \in I$ and $\phi(c_{\kappa,i}) = c_{\kappa',i}$ if $i \in J$.
It remains to verify that $\phi$ is isometric,
i.e.\ that $v(x_1 - x_2) = v(\phi(x_1) - \phi(x_2))$
for any $x_1, x_2 \in Y_\kappa$.

Suppose $x_1, x_2 \in Y_\kappa$ are given.
Let $i \in I$ be such that $x_1 \in D_{\kappa,i}$ or
$i \in J$ such that $x_1 = c_{\kappa,i}$. Choose $j$ analogously for $x_2$.
Then $\phi(x_1) - \phi(x_2) = x_1 - c_{\kappa,i} + c_{\kappa',i} - x_2 + c_{\kappa,j} - c_{\kappa',j}
= x_1 - x_2 - (c_{\kappa,i} - c_{\kappa,j}) + (c_{\kappa',i} - c_{\kappa',j})$, so it is enough to
show that
\begin{equation}\label{eq:isometrie}
v(x_1 - x_2) < v((c_{\kappa,i} - c_{\kappa,j}) - (c_{\kappa',i} - c_{\kappa',j}))
.
\end{equation}

We may suppose $i \ne j$; otherwise, this is trivial.
Now recall that
$c_{\kappa,i} - c_{\kappa,j}$ is $\delta$-uniform in $\kappa$ and
that $v(c_{\kappa,i} - c_{\kappa,j})$ is involved.
Suppose first that 
$v(c_{\kappa,i} - c_{\kappa,j})$ is constant.
Then we get $c_{\kappa,i} - c_{\kappa,j} \etwa
c_{\kappa',i} - c_{\kappa',j}$, so the right hand side
of \req{isometrie} is at least $v(c_{\kappa,i} - c_{\kappa,j}) +
\delta$.
If $x_1 = c_{\kappa,i}$ and $x_2 = c_{\kappa,j}$, then this implies \req{isometrie}
trivially.
If $x_1 \in D_{\kappa,i}$ and $x_2 = c_{\kappa,j}$, then apply
Lemma~\ref{lem:cellEasy} (1).
If $x_1 \in D_{\kappa,i}$ and $x_2 \in D_{\kappa,j}$, then apply
Lemma~\ref{lem:cellEasy} (2).

If $v(c_{\kappa,i} - c_{\kappa,j})$ is not constant, then
by \req{ifuSort}
both $c_{\kappa,i} - c_{\kappa,j}$ and $c_{\kappa',i} - c_{\kappa',j}$
have valuation at least $\lambda + 2\delta$, so we have to check
$v(x_1 - x_2) < \lambda + 2\delta$.
If $x_1 = c_{\kappa,i}$, then this
follows from $x_2 \notin B_{\kappa,i}$.
If $x_1 \in D_{\kappa,i}$, then $x_1 \notin B_{\kappa,i}$,
i.e.\ $v(x_1 - c_{\kappa,i}) < \lambda + \delta$, and the claim follows from
Lemma~\ref{lem:cellEasy} (1).
\end{proof}

\subsection{Proof for definable subsets of %
\texorpdfstring{$\Qp^2$}{\041\032p\texttwosuperior}}
\label{subsect:proofZp2}

We are now ready to prove that
if $X$ is a definable subset of $\Qp^{2}$, then the tree of
$X$ is of level $2$. To finish the proof of
Theorem~\ref{thm:mainZp2} we moreover need that
if $\dim X \le 1$, then the tree is of level $\dim X$;
this is included in Theorem~\ref{thm:main1dim}, which
we will prove in the next subsection.

\begin{proof}[Proof for two-dimensional subsets of $\Qp^2$]

Suppose that $X \subset \Qp^{2}$ is definable. Our goal
is to prove that $\Tr(X)$ is a tree of level $2$.
We use Lemma~\ref{lem:compact}, i.e.\ it is enough to show that for
any $(x_0,y_0) \in \Zp^2$ and for sufficiently large $\lambda, \mu, \rho$,
the trees on the corresponding garlands are piecewise of level $1$.
We suppose without loss $(x_0,y_0) = (0,0)$.

For the remainder of the proof fix a garland $G$ for $(0, 0), \lambda, \mu, \rho$.
At several places, we will suppose $\lambda, \mu, \rho$ to be sufficiently large;
of course the meaning of ``sufficient'' must not depend
on $G$ (as augmenting $\mu$ and $\rho$ augments the number of garlands).
Indeed, $\lambda, \mu, \rho$ will
only depend on two cell decompositions of $X$: a normal one and one
with coordinates exchanged.

For $\kappa \in M := M(G)$, let $G_{\kappa}$ be the corresponding component of $G$.
Recall that $G_{\kappa} = \ball{p^\kappa\cdot (x_G,y_G), \kappa + \mu}$
for some $(x_G,y_G) \in \Zp^2$ with $v(x_G,y_G) = 0$.
We may suppose $v(x_G) = 0$; otherwise, exchange coordinates.

Denote by $H$ the projection of $G$ onto the first coordinate and
by $H_{\kappa} = \ball{p^\kappa x_G, \kappa + \mu}$ the projections of the components $G_\kappa$.
As $v(x_G) = 0$, $H$ is a garland with components $H_\kappa$.
Denote by $B_\kappa = \ball{p^\kappa y_G, \kappa + \mu}$ the projection of
$G_\kappa$ onto the second coordinate.
For $x \in H$, let $X_x := \{y\in \Qp \mid (x,y) \in X\}$ be
the fiber of $X$ at $x$.

Our goal is to compute $\Tr_G(X)$.
We will verify that Corollary~\ref{cor:contTree} can be applied to
each set $G_\kappa \cap X$, yielding that $\Tr_{G_{\kappa}}(X)$ is
isomorphic to $\Tr(\Zp) \times \Tr_{B_\kappa}(X_{x_\kappa})$, where
$x_{\kappa} := p^\kappa x_G \in H_\kappa$. We will moreover
verify that Corollary~\ref{cor:1dimParam} can be applied to the sets
$X_{x_\kappa}$ and the balls $B_\kappa$ (where $\kappa$ runs through $M$).
This implies that the map $\kappa \mapsto \Tr_{B_\kappa}(X_{x_\kappa})$
is piecewise a level $1$ tree. Thus $T_G(X)$
satisfies the prerequisites of Lemma~\ref{lem:compact}, and we are done.

\medskip

Before we attack the prerequisites of the two corollaries, let us
have a closer look at the set $X$ and fix some notation.
Choose a cell decomposition
such that $X$ is the union of cells. We may
suppose that the exponents of all cells are equal to one
single $e_0 \in \bbN$. Fix once and for all $\delta := \delta(e_0)$ as in
Lemma~\ref{lem:cellEasy}.
By Lemma~\ref{lem:1dimLokal}, we may suppose that $H$ is contained in
one single base cell $D_0 \subset \Qp$.

In the remainder of the proof, $C$ will be a cell contained in $X$ and having
base $D_0$; we will denote its bounds and center by
$\alpha$, $\beta$ and $c$, respectively, and its fiber at $x \in H$
by $C_x$. For any $x \in H$, these fibers $C_x$ form a cell decomposition of $X_x$.
Occasionally we will need a second cell $C'$ (also contained in $X$ and having
base $D_0$), with bounds, center and fiber $\alpha'$, $\beta'$, $c'$
and $C'_x$.

We use Proposition~\ref{prop:defPuiseux} and Corollary~\ref{cor:defPuiseux}
to control $\alpha$, $\beta$ and $c$:
for $\lambda, \mu, \rho$ sufficiently large, the bounds $\alpha(x)$ and $\beta(x)$
only depend on $\kappa = v(x)$,
and this dependence is linear. Moreover, we can choose an $e$-th root
on $H$ and write the center as a convergent series
\[
c(x) = \sum_{i \in \bbZ} c_i \sqrt[e]{x}^i
,
\]
where $c_i = 0$ for $i \ll 0$, and
where $c_i$ may lie in $\aclQp$, but $c_i \sqrt[e]{x}^i \in \Qp$ for
any $x \in H$ and any $i \in \bbZ$ by Lemma~\ref{lem:series} \rit{sQp}.
We may suppose that $e$ does not depend on the cell $C$; otherwise,
take the least common multiple of all $e$.
For the remainder of the proof, we keep an $e$-th root on $H$ fixed.

Let $\iota$ be minimal such that $c_\iota \ne 0$ in the above series.
By further enlarging $\lambda$,
we may suppose $c(x) \etwa c_{\iota}\sqrt[e]{x}^\iota$ for all $x \in H$.
The same argument also applies to $f(x) := c(x) - c'(x)$
and to $f(x) := c(x) - \frac{y_G}{x_G}x$: we may assume that
for each of the (finitely many) functions $f$ mentioned here, there exist
$a \in \aclQp$ and $\iota \in \bbZ$ such that
$f(x) \etwa a\sqrt[e]{x}^\iota \in \Qp$ for all $x \in H$.

\medskip

We now verify the prerequisites of Corollary~\ref{cor:1dimParam},
i.e.\ we have to verify that the cell decomposition $C_{x_\kappa}$
of $X_{x_\kappa}$ satisfies the uniformness properties in $\kappa$.
It is clear that only the bounds and the centers depend on $\kappa$,
and we already ensured that the bounds are linear in $\kappa$.
It remains to verify that the functions
$\kappa \mapsto c(x_\kappa) - c'(x_\kappa)$ and
$\kappa \mapsto c(x_\kappa) - p^{\kappa}y_G$ are $\delta$-uniform.

Choose $a \in \aclQp$ and $\iota \in \bbZ$ such that
$c(x_\kappa) - c'(x_\kappa) \etwa a\sqrt[e]{x_\kappa}^\iota = a\sqrt[e]{p^\kappa x_G}^\iota$
and fix any $\kappa_0 \in M$. Then we can write any $\kappa \in M$
as $\kappa = \kappa_0 + e\nu$ for some $\nu \in \Gamma$. By uniformity
of the choice of roots on $H$, we have
$a\sqrt[e]{p^\kappa x_G}^\iota = p^{\iota \nu}a\sqrt[e]{p^{\kappa_0}x_G}^\iota$.
As only $\nu$ depends on $\kappa$, this yields $\delta$-uniformity of
$c(x_\kappa) - c'(x_\kappa)$.
The same argument applies to
$c(x_\kappa) - p^{\kappa}y_G = c(x_\kappa) - \frac{y_G}{x_G}x_\kappa
\etwa a\sqrt[e]{x_\kappa}^\iota$.
 
\medskip

The last remaining task is the
verification of the prerequisites of Corollary~\ref{cor:contTree}.
Fix $\kappa \in M$ and suppose we are given $x_1, x_2 \in H_\kappa$.
We have to find a bijective isometry
$\phi\colon X_{x_1} \cap B_\kappa \to X_{x_2} \cap B_\kappa$
satisfying $v(\phi(y) - y) \ge v(x_2 - x_1)$. We will define
$\phi$ on each cell $C_{x_1}$ separately.
However, first we have to get rid of some cells:
we claim that we can suppose
\begin{equation}\label{eq:centerFlat}
v(c(x)) \ge v(x)
\end{equation}
for all $x \in H$.

As $c(x) \etwa a\sqrt[e]{x}^\iota$ for some $a \in \Qp, \iota \in \bbZ$,
we may enlarge $\lambda$ such that \req{centerFlat} either holds for all
$x \in H$ or for no $x \in H$. Suppose that it does not hold.
We prove that then $C \cap G_\kappa$ is either empty or equal to $G_{\kappa}$
(i.e.\ either we may ignore $C$ or $T_{G_\kappa}(X)$ is trivial).
We have to check that for $(x_1, y_1), (x_2, y_2) \in G_{\kappa}$,
$y_1 \in C_{x_1}$ if and only if $y_2 \in C_{x_2}$.
The cell $C_{x_2}$ is just a shift of $C_{x_1}$
(the bounds $\alpha$ and $\beta$ only depend on $\kappa$),
so in view of Lemma~\ref{lem:cellEasy} (1)
it is enough to verify $y_1 - c(x_1) \etwa y_2 - c(x_2)$. But indeed,
we have
$v(c(x_1)) < \kappa \le v(y_1)$, so $v(y_1 - c(x_1)) = v(c(x_1)) < \kappa$,
and the claim follows from
$v(y_1 - y_2) \ge \kappa + \delta$ (which is true if we choose $\mu \ge
\delta$)
and $c(x_1) \etwa a\sqrt[e]{x_1}^\iota \etwa a\sqrt[e]{x_2}^\iota \etwa
c(x_2)$ (which follows from Lemma~\ref{lem:chooseRoot} if we choose $\mu \ge \delta + v(e)$).

Now let us define $\phi$. For $y \in X_{x_1}$, let $C$ be the cell such that
$y \in C_{x_1}$ and set $\phi(y) := y - c(x_1) + c(x_2)$. It is clear that
this defines a bijection $X_{x_1} \to X_{x_2}$, and it remains to verify
that $\phi$ is an isometry, restricts to a bijection
$X_{x_1} \cap B_\kappa \to X_{x_2} \cap B_\kappa$ and satisfies
\begin{equation}\label{eq:stetiso}
v(\phi(y) - y) \ge v(x_2 - x_1)
.
\end{equation}

Restricting to $B_\kappa$ is in fact a special case of
Equation~\req{stetiso}, as $B_\kappa$ is a ball of radius $\kappa+\mu
\le v(x_2 - x_1)$. By \req{centerFlat}, we may apply Lemma~\ref{lem:series} \rit{sflat},
which (after enlarging $\lambda$) implies \req{stetiso} using
$\phi(y) - y = c(x_2) - c(x_1)$.

To check that $\phi$ is an isomerty, suppose $y \in C_{x_1}$ and
$y' \in C'_{x_1}$. If $C = C'$, then $\phi(y') - \phi(y) = y' - y$,
so there is nothing to do. Otherwise we have
$v(\phi(y') - \phi(y)) = v(y' - c'(x_1) + c'(x_2) - y + c(x_1) - c(x_2))$,
so it is enough to
check
\begin{equation}\label{eq:isomY}
v(y' - y) < v\big((c'(x_1) - c(x_1)) - (c'(x_2) - c(x_2))\big)
.
\end{equation}

We have $c'(x_1) - c(x_1) \etwa a\sqrt[e]{x_1}^\iota$
and $c'(x_2) - c(x_2) \etwa a\sqrt[e]{x_2}^\iota$
for suitable $a$ and $\iota$.
Choosing $\mu \ge \delta + v(e)$ yields
$\sqrt[e]{x_1}^\iota \etwa \sqrt[e]{x_2}^\iota$, so
$c'(x_1) - c(x_1) \etwa c'(x_2) - c(x_2)$, i.e.\ the right hand side
of Equation~\req{isomY} is at least $v(c'(x_1) - c(x_1)) + \delta$.
But $y$ and $y'$ are contained in two disjoint cells, so
Lemma~\ref{lem:cellEasy} (2) yields $v(y' - y) < v(c'(x_1) - c(x_1)) + \delta$.
This proves isometry and finishes the proof of the theorem.
\end{proof}

\subsection{Proof for $1$-dimensional definable sets}
\label{subsect:proof1dim}

The proof of the conjecture for $1$-dimensional definable sets is in
many aspects just a simplification of the proof for subsets of $\Qp^2$,
so we will be less detailed. A level $0$ version of
Proposition~\ref{prop:1dimParam} will be build directly into the proof.

\begin{proof}[Proof of Theorem~\ref{thm:main1dim}]
If $X \subset \Qp^n$ is $0$-dimensional, then it is finite, so it is clear that
$\Tr(X)$ is a tree of level $0$.
Now let $X \subset \Qp^n$ be $1$-dimensional definable.
In this proof, we will view $\Qp^n$ as $\Qp \times \Qp^{n-1}$
and write elements as $(x, \ytup)$; all underlined variables
will be $(n-1)$-tuples.

By Lemma~\ref{lem:compact}, it is enough to show that for
any $(x_0,\ytup_{0}) \in \Zp^n$ and for sufficiently large $\lambda, \mu, \rho$,
the trees on corresponding garlands are of level $0$.
Without loss suppose $(x_0,\ytup_{0}) = 0$.
Again we fix a corresponding garland $G$
with components $G_{\kappa} = \ball{p^\kappa\cdot (x_G, \ytup_{G}), \kappa + \mu}$
for some $(x_G, \ytup_{G}) \in \Zp^n$ with $v(x_G, \ytup_{G}) = 0$.
By permuting coordinates, we may suppose $v(x_G) = 0$.

We use the same notation as in the proof for subsets of $\Qp^2$:
$H$ and $H_\kappa$ are the projections of $G$ and $G_\kappa$
onto the first coordinate,
$B_\kappa = \ball{p^\kappa \ytup_{G}, \kappa + \mu}$ is the projection of
$G_\kappa$ onto the remaining coordinates, and
for $x \in H$, $X_x := \{\ytup\in \Qp^{n-1} \mid (x,\ytup)
\in X\}$ the fiber of $X$ at $x$.
Again $H$ is a garland with components $H_\kappa$.

We will again apply Corollary~\ref{cor:contTree} to the sets
$G_\kappa \cap X$ to get $\Tr_{G_\kappa}(X) \cong \Tr(\Zp) \times \Tr_{B_\kappa}(X_{x_\kappa})$,
where $x_\kappa := p^\kappa x_G$. Moreover, we will show that
$\kappa \mapsto \Tr_{B_\kappa}(X_{x_\kappa})$ is piecewise
of level $0$; then the theorem follows.

Choose a cell decomposition of $\Qp^n$ such that $X$ is the union of cells, and
suppose that $C$ is a ``relevant'' cell, i.e.\ contained in $X$ and
intersecting $G$. Denote by $D_0 \subset \Qp$ the ``final base'' of $C$, i.e.\ iterate
taking the base $n-1$ times. We may suppose
$H\subset D_0$, so all relevant cells have the same final base $D_0$, and
moreover $\dim D_0 = 1$.

As $C$ is $1$-dimensional, it is the graph of a definable
function $\ctup \colon D_0 \to \Qp^{n-1}$. In this proof, by
the ``center'' of $C$ we shall mean this function $\ctup$.
By Proposition~\ref{prop:defPuiseux}, we may enlarge $\lambda,\mu,\rho$,
choose an $e$-th root on $H$ and then write the center as
\begin{equation}\label{eq:ctupseries}
\ctup(x) = \sum_{i\in \bbZ} \ctup_i\sqrt[e]{x}^i
.
\end{equation}
As $v(x - p^\kappa x_G) \ge \kappa + \mu$ for $x \in H_\kappa$, we have
$B_\kappa = \ball{p^\kappa x_G\frac{\ytup_{G}}{x_G}, \kappa + \mu}
= \ball{x\frac{\ytup_{G}}{x_G}, \kappa + \mu}$, so $\ctup(x) \in B_\kappa$
if and only if $v\big(\ctup(x) - x\frac{\ytup_G}{x_G}\big) \ge \kappa + \mu$.
Using \req{ctupseries}, this does not depend on $x$ if $\kappa \gg 0$,
so after enlarging $\lambda$ and removing irrelevant cells, we have
$\ctup(x) \in B_\kappa$ for all $x \in H_\kappa$ and all $\kappa \in M$.

Let $\ctup'$ be the center of a second cell $C'$.
By Corollary~\ref{cor:defPuiseux} we may suppose that
$v(\ctup(x) - \ctup'(x))$
only depends on $\kappa = v(x)$ and is linear
in $\kappa$. Let us call the induced functions
$v(x) \mapsto v(\ctup(x) - \ctup'(x))$
the ``involved functions''.

To show that $M \to \Trees, \kappa \mapsto \Tr_{B_\kappa}(X_{x_\kappa})$
is piecewise of level $0$, we partition $M$ into definable pieces $M'$
in such a way that for any two involved functions $\ifu_1, \ifu_2$,
the truth values of $\ifu_1 \lesseqqgtr \ifu_2$ are constant
on each piece $M'$.
The tree $\Tr_{B_\kappa}(X_{x_\kappa})$
has one infinite path for each
center $\ctup(x_\kappa)$, and the depths of the bifurcations are given by
$v(\ctup(x_G) - \ctup'(x_G))$. The partition of $M$ ensures that
the overall structure of $\Tr_{B_\kappa}(X_{x_\kappa})$ is constant
on each piece $M'$, and linearity of the involved functions yields
linearity of the lengths of the bones on each piece.

It remains to verify the prerequisites of Corollary~\ref{cor:contTree}.
For $\kappa \in M$ and $x_1, x_2 \in H_\kappa$, we use the bijection
$\phi\colon X_{x_1} \cap B_\kappa \to X_{x_2} \cap B_\kappa$
sending $\ctup(x_1)$ to $\ctup(x_2)$.
This is an isometry as $x \mapsto v(\ctup(x) - \ctup'(x))$ is
constant on $H_\kappa$. To get
$v(\ctup(x_2) - \ctup(x_1)) \ge v(x_2 - x_1)$ we apply
Lemma~\ref{lem:series} \rit{sflat} to each coordinate of $\ctup$;
the prerequisite $v(\ctup(x)) \ge v(x)$ follows from
$\ctup(x) \in B_\kappa$.
\end{proof}

\section{Possible generalizations}
\label{sect:open}

\subsection{Skeletal cell decompositions of trees}

The main conjecture can be generalized to a kind
of cell decomposition of trees in the following sense.
Consider $\Tr(\Zp^n)$ as an imaginary sort of our language:
\[
\raisebox{1.5ex}{$\Tr(\Zp^n) = (\Zp^n \times \Gamma)$} \!\!\bigg/
  \begin{aligned}
  &(\xtup, \lambda) = (\xtup', \lambda)\\[-0.5ex]
  &\text{if } v(\xtup - \xtup') \ge \lambda
  .
  \end{aligned}
\]
Then for any definable set $X\subset \Zp^n$, $\Tr(X)$ is a
definable subset of $\Tr(\Zp^n)$. Suppose we have an
isomorphism between $\Tr(X)$ and a tree constructed
out of a level $d$ tree datum; I will call this an \emph{iterated skeleton}
for $\Tr(X)$. Now let us add more branches to this iterated
skeleton in such a way that afterwards each node has exactly
$p^n$ children: enlarge the finite trees $\fintree$ at the beginning
of side branches, and add side branches to the iterated side trees which
before were of level $0$. The result is an iterated skeleton of level $n$ for
$\Tr(\Zp^n)$ which
is, in a certain sense, compatible to $\Tr(X)$. It seems plausible
that such a compatible iterated skeleton of $\Tr(\Zp^n)$
should exist for arbitrary definable sets $Y \subset \Tr(\Zp^n)$.
Let me make this more precise.

Let $D$ be a tree datum and let $\tree$ be the tree constructed
out of $D$.
Suppose that $\fintree$ is the finite tree appearing
in a side branch datum of $D$---either for
side branches of $\tree$ itself, or for side branches
of an (iterated) side tree. Suppose moreover that $w$ is a
node of $\fintree$. Then we define the set $C_{\fintree, w} \subset \tree$
of ``nodes coming from $w$''.
We would like to say that every node of $\tree$ lies in
exactly one set $C_{\fintree, w}$; to achieve this, we slightly
modify some definitions.

The only nodes of $\tree$ which
are not part of any set $C_{\fintree, w}$ are the ones on
side trees of level $0$. (Nodes on skeletons of trees of higher level
are roots of side branches.) Thus
we define a side branch of level $-1$ to be a finite tree $\fintree$
consisting only of a root,
and we let a tree of level $0$ be one with side branches of level $-1$
(as in Subsection~\ref{subsect:poincare}).
Now some nodes of $\tree$ appear in two sets $C_{\fintree, w}$:
if $w$ is a leaf of $\fintree$ and $\fintree$ belongs to a side
branch of level $\ge 0$, then the corresponding nodes of $\tree$
also appear as root of the first side branch of the side tree
attached to $w$; thus we forbid to take for $w$ a leaf of $\fintree$
unless $\fintree$ is a side branch of level $-1$.

In this way, an iterated skeleton of a tree $\tree$ yields
a partition of its nodes; let us call such a partition a
\emph{skeletal cell decomposition} of $\tree$, and let us call the sets
$C_{\fintree, w}$ \emph{skeletal cells}.
Now we can formulate a cell decomposition version of Conjecture~\ref{conj:main}:

\begin{conj}\label{conj:skelCell}
Suppose $Y \subset \Tr(\Zp^n)$ is definable. Then there exists
a skeletal cell decomposition of $\Tr(\Zp^n)$ such that $Y$
is a union of skeletal cells.
\end{conj}

This conjecture does not yet imply Conjecture~\ref{conj:main};
one would like to have a notion of
dimension for definable subsets of
$\Tr(\Zp^n)$ and then improve the statement to something
like ``$Y$ is a union of skeletal cells of level at most $\dim Y$''.

In the introduction, we mentioned a variant $\Trtil(V)$
of the tree of a variety
$V$, where the set of nodes at depth $\lambda$ consists of the
whole set $V(\bbZ/p^\lambda\bbZ)$.
These trees are definable, so they also fall in the scope of
this version of the conjecture.
Note that as for Conjecture~\ref{conj:main}, this
directly implies rationality of the associated Poincar\'e series:
the proof that trees of level $d$ have rational Poincar\'e series
directly generalizes to unions of skeletal cells, if one defines
the Poincar\'e series of a subset $Y \subset \Tr(\Zp^n)$ by
\[
P_{Y}(\sv) := \sum_{\lambda = 0}^\infty \#\{v \in Y \mid \depth_{\Tr(\Zp^n)}(v) = \lambda\} \cdot \sv^\lambda
.
\]

\subsection{Trees over other Henselian fields}

If $K$ is any Henselian field,
then one can define the
tree of a definable subset of $K^n$ in an analogue way as over $\Qp$
(though one needs a generalized
notion of tree if the valuation group is not discrete).
One cannot expect to get a nice statement on such trees
if the model theory of $K$ is not understood,
but there are several cases in which it is understood
and where a variant of the main conjecture would be interesting:
algebraically closed valued fields and
Henselian fields of characteristic $(0, 0)$.
Moreover, if the model theory is not understood, one
may still hope for a conjecture concerning trees of varieties.

The reason I think algebraically closed fields are interesting is
that there, trees should be simpler, and one might hope to first
prove a version of the conjecture in this case, before going back
to non-algebraically closed fields. Indeed,
over $\Qp$, we had different side branches depending on
the depth modulo some $\rho$. The reason for this was that not all roots exist, so
this phenomenon should disappear over algebraically closed fields.

Concerning Henselian fields $K$ of characteristic
$(0, 0)$, a good version of the conjecture there should imply a
uniform version of the conjecture over $\Qp$ for almost all $p$,
which in turn should imply rationality of the Poincar\'e series
``uniformly in $p$'', probably in the same sense as
it has been proven in \cite{DL:def}.
Let me make this precise, describing the hopes I have in this case.

Over $\Qp$, our trees were purely combinatorial; if the residue
field is not finite, then most nodes will just have infinitely
many children, so there is not much combinatorial information left.
Thus it will be necessary to add some additional structure to the
trees; probably the set of children of a node (or the appropriate equivalent
if the value group is not discrete) should be a definable
set over the residue field.
A tree datum $D$ in this setting should contain
formulas $\chi(\ytup)$ in the ring language, which describe the
sets of children of some nodes; for any valued field $K$, one then gets
an actual tree $\tree_{D,K}$ by
interpreting the formulas $\chi(\ytup)$ in the residue field of $K$.

Now suppose
that for any Henselian field $K$ of characteristic $(0,0)$ and any
formula $\phi(\xtup)$ (with $\xtup$ in the valued field sort),
we do not only have a tree datum $D$ describing
$\Tr(\phi(K))$, but moreover we can say this in a first order way:
there is a sentence $\psi$ which holds in $K$ and such that
for any other valued field $K'$, $K' \models \psi$ implies
that $D$ describes $\Tr(\phi(K'))$.
Then for any given formula $\phi(\xtup)$, by compactness there is
a finite set $\dats$ of tree data such that for any $K$ Henselian of
characteristic $(0,0)$,
there is a $D \in \dats$ describing $\Tr(\phi(K))$.
If we restrict ourselves to fields with value group
(elementarily equivalent to) $\bbZ$, then by Ax-Kochen-Er\v sov
$D$ will only depend on the residue field. Thus
we may unify all $D \in \dats$ to one single tree datum
$D_0$ which is valid for all $K$ by
incorporating the choice of $D$ into the formula
describing the children of the root.
By applying this to ultraproducts
of the fields $\Qp$, we get that $D_0$ also describes $\Tr(\phi(\Qp))$
for almost all $p$.

%
%
%
%
%

\def\cprime{$'$} \providecommand{\herrmannHook}{}

\bigskip

{\footnotesize
\noindent
Immanuel Halupczok\\
DMA\\
Ecole Normale Sup\'erieure\\
45, rue d'Ulm\\
75230 Paris Cedex 05---France\\
\emph{E-mail:} \texttt{math@karimmi.de}

}


\begin{thebibliography}{1}
\expandafter\ifx\csname url\endcsname\relax
  \def\url#1{\texttt{#1}}\fi
\expandafter\ifx\csname urlprefix\endcsname\relax\def\urlprefix{URL }\fi

\bibitem{Den:rat}
J.~Denef, The rationality of the {P}oincar\'e series associated to the
  {$p$}-adic points on a variety, Invent. Math. \textbf{77}~(1) (1984) 1--23.

\bibitem{SD:dim}
P.~Scowcroft, L.~van~den Dries, On the structure of semialgebraic sets over
  {$p$}-adic fields, J. Symbolic Logic \textbf{53}~(4) (1988) 1138--1164.

\bibitem{Nas:arc}
J.~F. Nash, Jr., Arc structure of singularities, Duke Math. J. \textbf{81}~(1)
  (1995) 31--38 (1996), a celebration of John F. Nash, Jr.

\bibitem{Den:cell}
J.~Denef, {$p$}-adic semi-algebraic sets and cell decomposition, J. Reine
  Angew. Math. \textbf{369} (1986) 154--166.

\bibitem{Clu:cell}
R.~Cluckers, Presburger sets and {$p$}-minimal fields, J. Symbolic Logic
  \textbf{68}~(1) (2003) 153--162.

\bibitem{Igu:locZeta}
J.-i. Igusa, An introduction to the theory of local zeta functions, Vol.~14 of
  AMS/IP Studies in Advanced Mathematics, American Mathematical Society,
  Providence, RI, 2000.

\bibitem{CL:mot}
R.~Cluckers, F.~Loeser, Constructible motivic functions and motivic
  integration, Invent. Math. \textbf{173}~(1) (2008) 23--121.

\bibitem{Eich:alg}
M.~Eichler, Einf\"uhrung in die {T}heorie der algebraischen {Z}ahlen und
  {F}unktionen, Lehrb\"ucher und Monographien aus dem Gebiete der exakten
  Wissenschaften, Mathematische Reihe, Band 27, Birkh\"auser Verlag, Basel,
  1963.

\bibitem{DL:def}
J.~Denef, F.~Loeser, Definable sets, motives and {$p$}-adic integrals, J. Amer.
  Math. Soc. \textbf{14}~(2) (2001) 429--469 (electronic).

\end{thebibliography}
\end{document}